\documentclass[11pt,a4paper]{amsart}

\usepackage[utf8]{inputenc}
\usepackage[T1]{fontenc}
\usepackage{lmodern}

\usepackage[margin=1in]{geometry}
\usepackage{amsmath, amssymb, amsthm}
\usepackage{mathrsfs}

\usepackage{graphicx}

\usepackage[size=tiny]{todonotes}

\usepackage{enumitem}

\usepackage[hidelinks]{hyperref}

\usepackage{color}

\newtheorem{theorem}{Theorem}[section]
\newtheorem{lemma}[theorem]{Lemma}
\newtheorem{proposition}[theorem]{Proposition}
\newtheorem{corollary}[theorem]{Corollary}
\newtheorem{conjecture}[theorem]{Conjecture}

\theoremstyle{definition}
\newtheorem{definition}[theorem]{Definition}

\theoremstyle{remark}
\newtheorem{remark}[theorem]{Remark}

\newlist{enuma}{enumerate}{1}
\setlist[enuma,1]{label=\textnormal{(\alph*)}, ref=\alph*}

\newcommand{\wt}{\widetilde}

\newcommand{\CC}{\mathbb{C}}

\newcommand{\NN}{\mathbb{N}}

\newcommand{\RR}{\mathbb{R}}

\newcommand{\fgg}{\mathfrak{g}}

\newcommand{\calB}{\mathcal{B}}
\newcommand{\calC}{\mathcal{C}}

\newcommand{\calF}{\mathcal{F}}

\newcommand{\calI}{\mathcal{I}}
\newcommand{\calK}{\mathcal{K}}

\newcommand{\calM}{\mathcal{M}}

\newcommand{\calT}{\mathcal{T}}
\newcommand{\calU}{\mathcal{U}}
\newcommand{\calV}{\mathcal{V}}

\DeclareMathOperator{\interior}{int}
\DeclareMathOperator{\im}{im}
\DeclareMathOperator{\vol}{Vol}

\DeclareMathOperator{\supp}{supp}

\DeclareMathOperator{\id}{id}

\newcommand{\Dens}{\operatorname{Dens}}

\DeclareMathOperator{\glob}{glob}
\DeclareMathOperator{\grad}{grad}
\DeclareMathOperator{\nc}{nc}
\renewcommand{\div}{\operatorname{div}}

\renewcommand{\S}{\operatorname{S}}
\newcommand{\C}{\operatorname{C}}

\renewcommand{\calC}{\mathscr{C}}

\newcommand{\calG}{\mathscr{G}}
\renewcommand{\calF}{\mathscr{F}}
\newcommand{\calX}{\mathscr{X}}
\newcommand{\calY}{\mathscr{Y}}

\renewcommand{\subset}{\subseteq}
\renewcommand{\b}{\overline}

\newcommand{\Curv}{\operatorname{Curv}}
\newcommand{\Val}{\operatorname{Val}}

\newcommand{\Z}{\operatorname{Z}}

\newcommand{\so}{\mathfrak{so}}
\newcommand{\gl}{\mathfrak{gl}}

\newcommand{\Sym}{\operatorname{Sym}}
\renewcommand{\SS}{\mathbb{S}}

\newcommand{\End}{\operatorname{End}}
\newcommand{\aff}{\operatorname{aff}}
\newcommand{\dir}{\operatorname{dir}}
\newcommand{\spn}{\operatorname{span}}

\renewcommand{\epsilon}{\varepsilon}

\usepackage[abbrev]{amsrefs}

\BibSpec{article}{%
	+{}{\PrintAuthors}  		{author}
	+{,}{ \textit}     		{title}
	+{,}{  \texttt } {eprint}
	+{,}{ }             		{journal}
	+{}{ \textbf}       		{volume}
	+{}{ \parenthesize} 		{date}
	+{}{, no. } 		{number}
	+{,}{ }      	      		{conference}
	+{,}{ }      	      		{book}
	+{,}{ }            		{pages}
	+{,}{ }            	 	{note}
	+{,}{ }            	 	{status}
	+{.}{}              {transition}
}

\BibSpec{book}{%
	+{}{\PrintAuthors}  {author}
	+{,}{ \textit}      {title}
	+{,}{ }             {publisher}
	+{,}{ }             {place}
	+{,}{ }             {date}
	+{.}{}              {transition}
}

\BibSpec{incollection}{%
	+{}  {\PrintAuthors}                {author}
	+{,} { \textit}                     {title}
	+{.} { }                            {part}
	+{:} { \textit}                     {subtitle}
	+{,} { \PrintContributions}         {contribution}
	+{,} { \PrintConference}            {conference}
	+{,} { }                            {booktitle}
	+{,} { pp.~}                        {pages}
	+{,}{ }       		{series}
	+{}{ \textbf}       		{volume}
	+{,} { }                            {publisher}
	+{,} { }                 {date}
	+{,} { }                            {status}
	+{,} { \PrintDOI}                   {doi}
	+{,} { available at \eprint}        {eprint}
	+{}  { \parenthesize}               {language}
	+{}  { \PrintTranslation}           {translation}
	+{;} { \PrintReprint}               {reprint}
	+{.} { }                            {note}
	+{.} {}                             {transition}
}

\title{Translation invariant curvature measures\\ of convex bodies}

\author{Jakob Schuhmacher}
\author{Thomas Wannerer}

\address{Friedrich-Schiller-Universit\"at Jena, Fakult\"at f\"ur Mathematik und Informatik, Institut f\"ur Mathematik, Inselplatz 5, 07743 Jena, Germany}

\email{jakob@familie-schuhmacher.de}
\email{thomas.wannerer@uni-jena.de}

\thanks{JS and TW were supported by 
	 the Deutsche Forschungsgemeinschaft (DFG, German Research Foundation)  WA 3510/3-1}

\begin{document}

\begin{abstract}
In a series of papers, Weil initiated the investigation of  translation invariant curvature measures of convex bodies, which include as prime examples Federer's curvature measures.   In this paper, we continue this line of research by introducing new tools to study curvature measures. Our main results suggest that the space of curvature measures, which is graded by degree and parity,  is highly structured: We conjecture that each graded component has length at most $2$ as a representation of the general linear group, and we prove this in degrees $0$ and $n-2$. 
Beyond this conjectural picture, our methods yield a characterization of Federer's curvature measures under weaker assumptions.
\end{abstract}

\maketitle

\section{Introduction}

\subsection{Background and principal results}

To study the curvature of singular spaces, two different but related approaches have been developed. One is to find  synthetic lower (or upper) bounds on curvature. The other is to view curvature as a generalized function. In his 1959 paper  
\cite{Federer:CurvatureMeasures}, Federer showed that one can make sense of the elementary symmetric functions of the principal curvatures  of smooth hypersurfaces as  measures for sets of positive reach. 

Federer's work has since been extended in several  directions. One of them concerns the construction of curvature measures for other classes of singular sets, see \cite{Fu:Kinematic,FPR:Kinematic,Zahle:Current,Fu:subanalytic,PokornyRataj:DC,RatajZahle:singular}. In another direction, curvature measures in ambient spaces  different from $\RR^n$ have been investigated \cite{FuW:Riemannian,BernigFu:Hig,BFS,BFS:PseudoRiemannian,BFS:pseudo-crofton}.
Federer's curvature  measures are also  fundamental to convex geometry \cite{Schneider:BM}.

Already Federer  \cite[Remark 5.17]{Federer:CurvatureMeasures} asked for an axiomatic  characterization of his curvature measures.   A first such result was obtained by Schneider \cite{Schneider:Curv}. Considering Federer's curvature measures as functions on the space of  convex bodies (nonempty compact convex subsets of $\RR^n$), his result is similar in spirit to Hadwiger's characterization of the intrinsic volumes. 

This perspective on curvature measures was further developed  by Kiderlen and Weil   in a series of papers \cite{KiderlenWeil:Measure,Weil:IGTF1,Weil:IGTF2}. Motivated by applications to stochastic geometry
\cite[Chapter 11]{SW:stochastic}, Weil calls  a function $\phi\colon \calK(\RR^n)\to \CC$ on the space of convex bodies in $\RR^n$ \emph{local} if it admits a \emph{kernel}, that is, if there exists a map $\Phi\colon \calK(\RR^n)\to M(\RR^n)$ to the space of complex Borel measures on $\RR^n$ satisfying $\Phi(K,\RR^n)= \phi(K)$ and the following properties.
\begin{enuma}
\item \emph{Locality}. Let $K,L$ be convex bodies in $\RR^n$. If there exists an open set $U$ such that $K\cap U=L\cap U$, then 
$$ \Phi(K,\beta)= \Phi(L,\beta)$$
for every Borel set $\beta\subset U$. 
\item \emph{Translation invariance}. $\Phi(K+x,\beta+x)= \Phi(K,\beta)$ for every convex body $K$,  Borel set $\beta$, and $x\in \RR^n$.
\item \emph{Continuity}. If a sequence $(K_i)$ of convex bodies converges to $K$, then $\Phi(K_i)\stackrel{w}{\longrightarrow}\Phi(K)$ in the sense of  weak convergence of measures.
\end{enuma}
Moreover, Weil imposed the condition that for every Borel set $\beta$, the function $K\mapsto \Phi(K,\beta)$ is Borel measurable. We will see below that this probability kernel property is in fact automatically implied by property (c).

The intrinsic volumes and Federer's curvature measures clearly fit into Weil's framework, and the  latter are the kernels of the former. 
Weil refrained from calling $\Phi$ a curvature measure and instead introduced the somewhat unspecific terminology of local functionals and their  kernels. However, in defense of Weil, let us emphasize that, at this level of generality, it is far from clear whether a map $\Phi$ satisfying properties (a)--(c) can, in any reasonable sense,  be related to the actual curvature of a convex body. One of the main insights of the present  work is that this is indeed the case. 
We conjecture, and prove this in several cases, that any such  $\Phi$ is obtainable from the second fundamental form of smooth convex bodies. We therefore propose the following definition.

\begin{definition}
	A map $\Phi\colon \calK(\RR^n)\to M(\RR^n)$ is called a \emph{translation invariant curvature measure} if it satisfies  properties (a) through (c). We denote by $\Curv(\RR^n)$ the vector space of translation invariant curvature measures.
\end{definition}

There is a natural subspace $\Curv^{sm}(\RR^n)\subset \Curv(\RR^n)$ of \emph{smooth curvature measures}, which are determined by the second fundamental form of smooth convex bodies. 
Let $S\RR^n= \RR^n\times S^{n-1}$ denote the sphere bundle of $\RR^n$, let $\pi\colon S\RR^n\to \RR^n$ denote the projection, and let $\Omega^*(S\RR^n)^{tr}$ denote the space of smooth differential forms invariant under all translations $(x,u)\mapsto (x+t,u)$. Following Bernig and Fu \cite{BernigFu:Hig},   a translation invariant curvature measure $\Phi$ is called \emph{smooth} if there exist a number $c\in \CC$ and a smooth differential form  $\omega\in \Omega^{n-1}(S\RR^n)^{tr}$  such that, for every convex body $K$ and every Borel set $\beta\subset \RR^n$,
$$\Phi(K,\beta)=  c \vol(K\cap \beta ) + \int_{\nc(K)} \mathbf{1}_{\pi^{-1}(\beta)}\, \omega.$$
Here  the normal cycle of $K$ is given by 
$$ 
\nc(K)=  \{ (x,u)\in S\RR^n\colon u \text{ is a normal of }K\text{ at } x\}.$$

It is not difficult to see that if $\Phi$ is  smooth and $c=0$, then there exists a smooth  family of polynomials $p_u=p^\Phi_{u}$ on the space of symmetric bilinear forms on $u^\perp$ such that, whenever $\partial K$ is $C^2$ and strictly positively curved,
$$ \Phi(K,\beta) = \int_{\beta \cap \partial K} p_{n_K(x)}(h_x^{\partial K}) \, dx.$$
Here $n_K(x)$ denotes  the outer unit normal at $x$  and $h_x^{\partial K}$   the second fundamental form of $\partial K$ at $x$. Thus the measure $\Phi(K, \,\cdot \,)$ is determined by the curvature of $\partial K$.

\medskip

As already observed by Weil, the vector space $\Curv(\RR^n)$ is graded by the degree of a curvature measure,
$$ \Curv(\RR^n)= \bigoplus_{k=0}^n \Curv_k(\RR^n).$$
 The grading can be further   refined by decomposing $\Curv_k(\RR^n) = \Curv^+_k (\RR^n) \oplus \Curv^-_k (\RR^n)$ into even and odd curvature measures. We refer the reader to Section~\ref{sec:weil} for precise definitions of these notions.
 
 The curvature measures of degrees $n-1$ and $n$ can be  described explicitly. As shown by Weil \cite{Weil:IGTF2}, the vector space $\Curv_n(\RR^n)$ is spanned by the localization of the Lebesgue measure 
 	$$ (K,\beta)\mapsto \vol(K\cap \beta).$$
 An explicit  description of $\Curv_{n-1}(\RR^n)$ is also possible, but slightly less trivial, see Theorem~\ref{thm:deg-n-1}. We will usually exclude the degrees  $n-1$ and  $n$ in the results below,  since statements  in these degrees  are significantly easier to verify.
 
 \medskip

Every  choice of a locally convex topology on  $M(\RR^n)$ induces  a topology  on $\Curv(\RR^n)$. 
The topologies of compact and bounded convergence on $M(\RR^n)$ lead to complete locally convex Hausdorff topologies on $\Curv(\RR^n)$. However, only the topology induced from  compact convergence is natural in the sense that it makes the action of the linear group $GL(n,\RR)$ on $\Curv(\RR^n)$ continuous. 
Henceforth, all topological statements will refer to this topology. 

\medskip

We now formulate  two conjectures concerning the structure of $\Curv_k(\RR^n)$. In this paper, we  confirm them in degrees $0$ and $n-2$. 
The picture suggested by these conjectures is that  $\Curv(\RR^n)$ is highly structured and well behaved. 
The first conjecture justifies the terminology of curvature measure for a map  $\Phi$ satisfying the properties (a)--(c) above. It gives a  precise meaning to the statement that the curvature of $K$ determines the measure $\Phi(K)$.  
Using Theorem~\ref{thm:deg-n-1}   for degree $n-1$, it is straightforward to see that the following statement  holds  for $k=n-1$ and $k=n$.

\begin{conjecture}\label{conj:sm-dense}
	Let $k\in \{0,\ldots, n-2\}$. The subspace of smooth curvature measures $\Curv_k^{sm}$ is dense in $\Curv_k(\RR^n)$. 
\end{conjecture}

We will show that Conjecture~\ref{conj:sm-dense} is implied by a remarkable property of the action of the general linear group on $\Curv(\RR^n)$ defined by
$$ (g\cdot \Phi)(K,\beta) = \Phi(g^{-1} K , g^{-1} \beta) $$ 
for all $g\in GL(n,\RR)$, convex bodies $K$, and Borel sets $\beta$.  
Curvature measures with total measure $0$ will play a central role  in our investigation. Let $\Z_k^\pm(\RR^n)\subset \Curv_k^\pm(\RR^n)$ denote the subspace of curvature measures satisfying  
$$ \Phi(K,\RR^n)=0$$ 
for all convex bodies $K$.

\begin{conjecture} \label{conj:irr} Let $k\in \{0,\ldots, n-2\}$. 
The representation of $GL(n,\RR)$ on the 
 invariant closed subspace $\Z_k^\pm(\RR^n)$ is irreducible.
\end{conjecture}

Using Theorem~\ref{thm:deg-n-1}   for degree $n-1$,  one immediately verifies the validity of the conjecture  in  degrees $n-1$ and $n$. The following theorem is one of the principal results of the present paper.

\begin{theorem}\label{thm:conj} Conjecture~\ref{conj:irr} holds in degrees $0$ and $n-2$. 
\end{theorem}

The concept of length of a representation $E$ refers to the length $m$ of a composition  series 
$$ E= E_0  \supsetneq E_1 \supsetneq \cdots \supsetneq E_m =\{ 0\},$$
where $E_i$ is a closed invariant subspace and $E_i/E_{i+1}$ is irreducible. Under certain conditions, which are satisfied for $\Curv(\RR^n)$, one can show that all composition series of $E$ have the same length. 

From Conjecture~\ref{conj:irr} one can also obtain information about the action of  $GL(n,\RR)$ on the entire space  $\Curv(\RR^n)$. Namely, in combination with  Alesker's irreducibility theorem for translation invariant valuations (see Section~\ref{sec:prelim}), we obtain the following result.

\begin{corollary}
	Assuming Conjecture~\ref{conj:irr}, the representation of $GL(n,\RR)$ on $\Curv^\pm_k(\RR^n)$ has length at most $2$. 
\end{corollary}

Our basic proof strategy is to transform problems about curvature measures into problems about translation invariant continuous valuations on convex bodies. This will allow us to exploit the  theory of such valuations, which has undergone  dramatic development following the work of Alesker in the early 2000s, see \cite{Alesker:Irreducibility,Alesker:Product,Alesker:manifoldsI,Alesker:barcelona,Alesker:kent,BKW:HR}. 

Let $\Val(\RR^n)$ denote the space of translation invariant continuous valuations on $\RR^n$, that is, the continuous and translation invariant functions 
$\phi\colon \calK(\RR^n)\to \CC$ satisfying 
$$ \phi(K\cup L)= \phi(K) + \phi(L) - \phi(K\cap L)$$
whenever $K\cup L$ is convex. Similarly, denote by $\Val(\RR^n, \CC^n)$ the space of  continuous translation invariant valuations with values in $\CC^n$. These spaces of valuations are graded by degree and parity and equipped with a natural action of the general linear group. 

Theorem~\ref{thm:conj} will follow  from the following statement.

\begin{theorem}\label{thm:Val2} 
	The $GL(n,\RR)$ representations 
	$\Val_1^{\pm}(\RR^n,\CC^n)$ and  $\Val_{n-1}^{\pm}(\RR^n,\CC^n)$ have  length $2$.
\end{theorem}

Valuations of degrees $1$ and $n-1$  admit a description in terms of homogeneous functions on $\RR^n$. This will allow us to reduce the proof of Theorem~\ref{thm:Val2} to the purely representation-theoretic problem  of determining the $GL(n,\RR)$ module structure of the spaces of homogeneous functions on $\RR^n$ with values in $\CC^n$.   A key geometric insight then allows us to deduce Theorem~\ref{thm:conj} from Theorem~\ref{thm:Val2}. As we show in Section~\ref{sec:kernel}, there  exists a  $GL(n,\RR)$ monomorphism $B\colon \Z_k^\pm(\RR^n)\to \Val_{k+1}^\pm(\RR^n,\CC^n)$. The image of this map is  not dense for $k\leq n-2$. Consequently, if $\Val_{k+1}^\pm(\RR^n,\CC^n)$ has length $2$, then $\Z_k^\pm(\RR^n)$ must be irreducible.

\subsection{Applications} 

We present two applications of the theory of translation invariant curvature measures developed in this paper. As a first application, we streamline the axiomatic framework underlying Weil's theory of local functionals by showing that every translation invariant curvature measure is  automatically a probability kernel. The second application concerns the characterization of Federer's curvature measures, where our results allow us to remove the hypothesis of nonnegativity from Schneider's theorem \cite{Schneider:Curv}.

\subsubsection{Weil's theory of local functionals and their application to stochastic geometry}
As already mentioned above,   Weil's terminology differs slightly from ours. The central objects of \cite{Weil:IGTF2} are local functionals and their kernels $\Phi$. The latter is a continuous translation invariant curvature measure $\Phi$ in our sense, with the additional property that for every Borel set $\beta$, the function
\begin{equation}\label{eq:transition} K \mapsto \Phi(K,\beta)\end{equation}
is measurable. This technical property plays a role in the context of the applications to Poisson particle processes and Boolean models, which were an important source of motivation for Weil's work. We show that  \eqref{eq:transition} is in fact automatically satisfied for curvature measures. 

\begin{proposition}\label{prop:baire}
	Let $\Phi\in \Curv(\RR^n)$. Then, for every bounded  Borel function  $f$, the function
	$$ K\mapsto \Phi(K,f)$$
	is Borel measurable.
\end{proposition}

\subsubsection{Characterization of Federer's curvature measures}
Let $\C_0,\ldots, \C_n$ denote Federer's curvature measures. Restricted to convex bodies in $\RR^n$, they can be regarded as elements of $\Curv(\RR^n)$.  
A curvature measure $\Phi\in \Curv(\RR^n)$ is called  $SO(n)$ invariant if   $$k\cdot \Phi = \Phi\quad \text{for all }k\in SO(n).$$
With this terminology, and using the result of Weil \cite{Weil:IGTF2} that every curvature measure satisfies the valuation property, Schneider's theorem \cite{Schneider:Curv} may be formulated as follows.

\begin{theorem}[Schneider--Weil] Let $\Phi\in \Curv(\RR^n)$ be an $SO(n)$ invariant curvature measure. If $\Phi(K)$ is a nonnegative measure for every $K$, then there exist nonnegative numbers $a_0,\ldots, a_n$ such that 
	$$ \Phi = \sum_{i=0}^n a_i \C_i.$$
\end{theorem}

A somewhat curious defect of this result is that it assumes that $\Phi(K)$ is a nonnegative measure. Schneider needs this assumption because his proof requires a characterization of the spherical Lebesgue measure as an $SO(n)$ invariant simple valuation on spherical convex sets. This, however, is a well-known and  long-standing open problem (discussed at length in the final chapter of \cite{KlainRota}). Under the additional assumption of nonnegativity, Schneider obtains such a characterization of spherical Lebesgue measure.

Our approach allows us to characterize linear combinations of Federer's curvature measures without assuming nonnegativity.
\begin{theorem}
\label{thm:federer}	
Let $\Phi\in \Curv(\RR^n)$ be an $SO(n)$ invariant curvature measure. Then there exist numbers  $a_0,\ldots, a_n$ such that 
$$ \Phi = \sum_{i=0}^n a_i \C_i.$$
\end{theorem}

A closely related, but more challenging problem is to  characterize the tensor-valued curvature measures of Hug and Weis  \cite{HugWeis:kinematic,HugWeis:crofton}. 
A confirmation of Conjecture~\ref{conj:irr} would resolve this problem in the translation invariant case, as it would allow us to remove the assumption of smoothness in the results of Saienko \cite{Saienko}.

\section{Preliminaries}

\label{sec:prelim}

\subsection{Support measures and the normal cycle of convex bodies}

We denote by $S\RR^n= \RR^n\times S^{n-1}$ the sphere bundle of $\RR^n$. A classical construction in convex geometry associates to each convex body $K\subset \RR^n$ finite Borel measures
$$ \Theta_0(K, \cdot ),\ldots , \Theta_{n-1}(K,\cdot)$$
 on the sphere bundle called the support measures of $K$; see \cite[Chapter 4]{Schneider:BM}.
  Closely related to these are the  curvature and area measures of $K$ defined by 
 $$ \C_i(K,\beta)= \Theta_i(K, \beta \times S^{n-1}),  \quad \S_i(K,\gamma) =  \Theta_i(K,\RR^n\times \gamma),$$
 for Borel sets $\beta\subset \RR^n$ and $\gamma\subset S^{n-1}$.  
The following lemma  collects the basic properties of support measures. 

\begin{lemma}
	\begin{enuma}
		\item If $K\cap U= L\cap U$ for some open set $U$, then 
		$$ \Theta_i(K, \eta)= \Theta_i(L,\eta) $$
		for every Borel set $\eta\subset U\times S^{n-1}$. 
		
		\item If $\rho$ is an isometry of $\RR^n$, then, for all Borel sets $\eta$ and convex bodies $K$, 
		$$ \Theta_i(\rho(K) ,\wt\rho(\eta))= \Theta_i(K,\eta),$$
		where $\wt\rho(x,u)= (\rho(x), \rho_0 u)$ and $\rho_0$ denotes the linear isometry associated with $\rho$. 
		\item If $K_j\to K$ in the Hausdorff metric, then $\Theta_i(K_j)\stackrel{w}{\longrightarrow} \Theta_i(K)$.
	\end{enuma}
\end{lemma}

Although initially defined on all of $S\RR^n$, the support of $\Theta_i(K,\cdot)$ is  contained in the normal cycle of $K$, defined as 
\begin{equation}\label{eq:nc}  \nc( K)= \{ (x,u)\in S\RR^n\colon u \text{ is a normal of }K\text{ at } x\}.\end{equation} 
The classical construction of the support measures is based on a local Steiner formula.
The theory of the normal cycle, developed by  Fu~\cite{Fu:IGRegularity,Fu:Kinematic} as a generalization of the work of Federer \cite{Federer:CurvatureMeasures} and Z\"ahle \cite{Zahle:Current},  provides an alternative approach. While technically  more demanding,  it  allows a transparent and uniform description of the support, area, and curvature   measures of $K$   and opens the door to a  theory of curvature measures for more general sets and ambient spaces \cite{FuW:Riemannian,FPR:Kinematic}.
 
The construction of the normal cycle does not present any difficulties in the convex case,  since \eqref{eq:nc} is a  closed oriented Lipschitz submanifold of $S \RR^n$. Consequently, smooth differential forms on $S\RR^n$ may be integrated over $\nc(K)$, and the normal cycle can be regarded as a current. 

In view of compatibility with linear transformations,  it is  more natural and convenient to work with the cosphere bundle $S^*\RR^n =\RR^n \times S^*(\RR^n)$ of $\RR^n$, where the cosphere $S^*(\RR^n)$ is the set of equivalence classes of nonzero linear functionals $\xi\colon \RR^n\to \RR$ for the equivalence relation $\xi\sim \xi'$ if and only if $ \xi' = a \xi$ for some $a>0$. The conormal cycle  of $K$ is  then defined as the subset of $S^*\RR^n$ of conormals to $K$. By abuse of notation, we also denote the conormal cycle by the symbol $\nc(K)$. 

Let  $T(x)= Ax+v$ be an invertible affine transformation of $\RR^n$. Then  $T$ is covered by the map $$\wt T\colon S^*\RR^n \to S^*\RR^n,\quad \wt T(x,[\xi]) = (T(x), [A^{-\top} \xi]),$$  and $\nc(T(K))= \wt T(\nc(K))$.  We denote by $\pi$ the canonical projection from the (co-)sphere bundle to $\RR^n$.

\begin{lemma}\label{lemma:prop-nc}
\begin{enuma}
	\item If $T\colon \RR^n\to \RR^n$ is an invertible affine transformation, then$$ \int_{\nc(TK )} \omega = \operatorname{sgn}(\det T)  \int_{\nc(K)} \wt T^*\omega. $$
	
\item 	If  $(K_j)$ is a sequence of convex bodies in $\RR^n$ converging to $ K$ in the Hausdorff metric, then, as $j \to  \infty$, 
	$$ \int_{\nc(K_j)} \pi^* f  \cdot \omega \rightarrow   \int_{\nc(K)} \pi^*f \omega$$
	for every $\omega\in \Omega^{n-1}(S\RR^n)$ and every $f\in C_b(\RR^n)$. 
	\item If $B$ is a bounded subset of $\RR^n$, then 
	$$ \sup_{K\subset B} \mathbf M(\nc(K)) <\infty,$$
	where $\mathbf M(T)$ denotes the mass of  a current $T$. 
	
\end{enuma}
\end{lemma}

For proofs of the above properties of the normal cycle we refer the reader to  \cite{Alesker:VMfdsIII}. A proof of (b) that is tailored to the convex case and yields a stronger result can be found in \cite[Theorem 2.20]{HugSchneider:tensor-survey}.

Finally, let us describe support measures and Federer's curvature measures in terms of the normal cycle. Consider on $\RR^n\times \RR^n$ with the standard coordinates $(x,y)$ the differential forms 
$$\kappa_i= \frac{1}{(n-1)!} \sum_{\sigma\in \mathfrak S_{n}}\operatorname{sgn}(\sigma) y_{\sigma(n)}  dx_{\sigma(1)} \wedge \cdots \wedge dx_{\sigma(i)} \wedge dy_{\sigma(i+1)} \wedge \cdots \wedge 
dy_{\sigma(n-1)}$$
and denote their restrictions to $S\RR^n$ by the same symbols. These differential forms are clearly invariant under orientation preserving isometries. 
For $i=0,\ldots, n-1$, 
\begin{align*}\C_i(K,\beta) &  = \int_{\nc(K)}  \mathbf{1}_{\pi^{-1}(\beta)}\, \kappa_i,\\ 
	\Theta_i(K,\eta) & =  \int_{\nc(K)} \mathbf{1}_{\eta}\, \kappa_i.
\end{align*}
Since these identities can  be immediately verified for polytopes, they hold  universally by continuity.

\subsection{Translation invariant valuations}

By a remarkable theorem of Weil, to be discussed in the next section, curvature measures satisfy the valuation property. This will allow us to prove powerful results about curvature measures by connecting them to the highly developed theory of translation invariant valuations on convex bodies. 

We denote by  $\Val(\RR^n)$ the vector space of continuous and translation invariant valuations $\phi\colon \calK(\RR^n)\to \CC$. A valuation $\phi$ is called homogeneous of degree $\alpha\in \RR$  if $\phi(tK)=t^\alpha \phi(K)$ holds for all convex bodies and positive numbers $t>0$. We denote by  $\Val_\alpha(\RR^n)$ the subspace of $\alpha$-homogeneous valuations. 

\begin{theorem}[McMullen \cite{McMullen:Valuations}]
	Every $\phi\in \Val(\RR^n)$ can be uniquely expressed as 
	$$ \phi= \phi_0+ \cdots + \phi_n$$ 
	with $\phi_i\in \Val_i(\RR^n)$.
\end{theorem}
In particular, McMullen's theorem shows that the degree of a translation invariant continuous valuation must be an integer between $0$ and $n$. 
Moreover, it  implies that for every bounded subset $B\subset \RR^n$ containing the origin in the interior, 
$$ \|\phi\|_B = \sup_{K\subset B} |\phi(K)|$$
defines a norm on $\Val(\RR^n)$. It is not difficult to see that this norm is in fact complete. 

Sometimes we will further decompose $\Val_i=\Val^+_i\oplus \Val^-_i$ by parity, where a  valuation is called even if $\phi(-K)= \phi(K)$ and it is called odd if $\phi(-K)=-\phi(K)$. 

It is straightforward to see that $\Val_0(\RR^n)$ is $1$-dimensional and consists of the  constant valuations. The corresponding result for valuations of maximal degree is less trivial and was proved by Hadwiger \cite[p. 79]{Hadwiger:Vorlesungen}.
\begin{theorem}[Hadwiger] \label{thm:hadwiger}
	$\Val_n(\RR^n)$ is $1$-dimensional and spanned by the Lebesgue measure.
\end{theorem}

Also valuations of degree $n-1$  admit  an explicit description.

\begin{theorem}[McMullen \cite{McMullen:Continuous}]\label{thm:McMullen-n-1}
Let $\phi\in \Val_{n-1}(\RR^n)$. Then there exists a continuous function 
$f\colon S^{n-1}\to \CC$ such that 
$$ \phi(K)= \int_{S^{n-1}} f(u)\, d\!\S_{n-1}(K,u).$$
	Moreover, the function $f$ is unique up to the addition of linear functionals. 
\end{theorem}

A valuation $\phi\in \Val(\RR^n)$ is called simple if $\phi(K)=0$ whenever $\dim K<n$. The following characterization of simple valuations was discovered by Klain \cite{Klain:short} and Schneider \cite{Schneider:simple}.

\begin{theorem}[Klain--Schneider] \label{thm:klain-schneider} A valuation $\phi\in\Val(\RR^n)$ is simple if and only if $\phi\in \Val^-_{n-1} \oplus \Val_n$. 
\end{theorem}

Representation theoretic methods have  played an important role in the development of valuation theory.
We  will review fundamental notions from representation theory in Section~\ref{sec:smooth}, and we postpone the discussion of more advanced topics to that section. Here we mention only  that the action of the general linear group on $\Val(\RR^n)$ is defined by 
$$ (g\cdot \phi)(K) = \phi(g^{-1}K).$$
A key property of this action is known as Alesker's irreducibility theorem.

\begin{theorem}[Alesker \cite{Alesker:Irreducibility}] \label{thm:alesker} Under the action of $GL(n,\RR)$, the spaces $\Val_i^\pm(\RR^n)$ are irreducible.
\end{theorem}
As usual in the context of infinite-dimensional representation theory, irreducible means that there exist no nontrivial invariant closed subspaces. 

Finally, a valuation $\phi\in \Val(\RR^n)$ is called smooth if there exist a number $c\in \CC$  and a differential form $\omega\in \Omega^{n-1}(S\RR^n)^{tr}$ such that, for all convex bodies $K$,  
\begin{equation}\label{eq:smooth-val}\phi(K)= c \vol(K)+ \int_{\nc(K)} \omega. \end{equation}
 The subspace of smooth valuations will be denoted by 
$\Val^{sm}$. It follows from Alesker's irreducibility  theorem that smooth valuations are dense.  We postpone the discussion of properties of smooth valuations to Section~\ref{sec:smooth} after we have reviewed the necessary background from representation theory, but mention already here that there  exists a counterpart to Theorem~\ref{thm:McMullen-n-1} for smooth valuations of degree $1$. A proof of an apparently stronger, but in fact  equivalent statement,  can be found in the appendix of \cite{Schuster:log}. We denote by $h_K\colon \RR^n\to \RR$, $h_K(x)= \sup_{y\in K} \langle x,y\rangle$, the support function of a convex body $K$. 

\begin{theorem}\label{thm:1-hom}
Let $\phi\in\Val_1^{sm}(\RR^n)$. Then there exists a unique smooth function $f\colon S^{n-1}\to \CC$ with the property $\int_{S^{n-1}} f(u) \, du =0$ such that 
$$ \phi(K)= \int_{S^{n-1}} h_K(u) f(u) \, du$$
holds for all convex bodies $K$. 
\end{theorem}

\section{Kiderlen--Weil decomposition and the valuation property}
\label{sec:weil}

In this section, we summarize the main results of the theory  developed by Kiderlen and Weil in \cite{KiderlenWeil:Measure, Weil:IGTF1, Weil:IGTF2}.
The central objects of \cite{Weil:IGTF2} are so-called  local functionals and their kernels $\Phi$, where the latter is a continuous translation invariant curvature measure in our sense with the additional property that for every Borel set $\beta$,
$$ K \mapsto \Phi(K,\beta)$$
is measurable. This technical property, which is in fact automatically satisfied, is irrelevant to the foundational results presented in this section.

The first theorem shows that the space of curvature measures is naturally graded. In this context, a curvature measure $\Phi\in \Curv(\RR^n)$ is called  \emph{homogeneous of degree} $\alpha\in \RR$ if 
$$ \Phi(t K, t \beta)= t^\alpha \Phi(K,\beta)$$
holds for all convex bodies $K$, Borel sets $\beta$, and numbers $t>0$. The subspace of curvature measures of degree $\alpha$ is denoted by $\Curv_\alpha$.

We will refer to the following result as the Kiderlen--Weil decomposition. It can be regarded as an analogue of the McMullen decomposition for translation invariant valuations, and it implies, in particular,  that the degree of a curvature measure is always an integer between $0$ and $n$. 
\begin{theorem}[Kiderlen--Weil \cite{KiderlenWeil:Measure}] \label{thm:KW}
	Every curvature measure $\Phi\in \Curv(\RR^n)$ can be uniquely written as 
	$$ \Phi= \Phi_0+ \cdots + \Phi_n$$
	with $\Phi_i\in \Curv_i$ for $i=0,\ldots, n$. 
	
\end{theorem}

The following remarkable result is central to our approach to curvature measures.

\begin{theorem}[Weil \cite{Weil:IGTF2}]  \label{thm:weil}	Every curvature measure $\Phi\in \Curv(\RR^n)$ is a valuation with values in $M(\RR^n)$: 
	$$\Phi(K\cup L) = \Phi(K)+ \Phi(L)- \Phi(K\cap L)$$
	whenever $K$ and $L$ are convex bodies such $K\cup L$ is convex.
\end{theorem}

In their paper  \cite{KiderlenWeil:Measure},  Kiderlen and Weil provide  additional information  about  the measure $\Phi(P)$ for polytopes  $P$. In particular, they show that an $i$-homogeneous curvature measure is supported on the $i$-dimensional faces of $P$. 
In the following theorem, $\calC_i$ denotes the set of  at most $i$-dimensional convex cones in $\RR^n$ with apex at the origin. The normal cone of a polytope $P$ at the face $F$ is denoted by $N(F,P)$ and the set of $i$-dimensional faces of $P$ is denoted by $\calF_i(P)$. 
\begin{theorem}
	\label{thm:polytopes}
	Let $\Phi\in \Curv_i(\RR^n)$. There exists a unique function $\varphi\colon \calC_{n-i}\to \CC$ such that for every polytope $P$
	$$ \Phi(P,\beta)= \sum_{F\in \calF_i(P)}   \varphi(N(F,P))  \vol_i(F\cap \beta).$$
The function $\varphi$ is a simple valuation in the sense that $\varphi(C)=0$ for cones of dimension less than $n-i$ and  
$$ \varphi(C_1\cup C_2)= \varphi(C_1)+ \varphi(C_2)-\varphi(C_1\cap C_2)$$ 
whenever $C_1$, $C_2$, and $C_1\cup C_2$ are in $\calC_{n-i}$.

\end{theorem}

We will frequently use the following weaker statement, which is an immediate consequence of the locality and translation invariance of $\Phi$; see \cite[p. 124]{Schneider:Curv}. As  a consequence, weak convergence can be replaced by vague convergence in the definition of continuity of a curvature measure.

\begin{lemma}\label{lemma:supp} Let $\Phi\in \Curv(\RR^n)$.  For every convex body $K$, the support of $\Phi(K)$ satisfies
	$$ \supp \Phi(K)\subset K.$$
\end{lemma}

\section{Topology on the space of curvature measures}

Kiderlen and Weil did not specify a topology on the space of curvature measures. 
Since a suitable topology is essential for our approach, we now continue the development of their theory by examining possible choices.

Let $X$ be a topological space and let $E$ be a topological vector space. Let $B(X,E)$ denote the  space of bounded functions $f\colon X\to E$. If the topology on $E$ is induced by a family of seminorms $\{p_\alpha\colon \alpha \in A\}$, then 
 a locally convex topology on $B(X,E)$ is induced by the seminorms
$$ f\mapsto   \sup_{x\in X} p_\alpha(f(x)), \quad \alpha \in A.$$

A curvature measure is by definition a function  $\Phi\colon \calK(\RR^n)\to M(\RR^n)$ satisfying certain additional properties.  
By the Riesz representation theorem, $M(\RR^n)$ can be identified with the topological dual of $C_0(\RR^n)$.  On $E'$, the topological dual of a locally convex space $E$, there are three topologies of special importance: the weak topology, the topology of compact convergence, and the topology of bounded convergence.  Let $B\subset \RR^n$ be a bounded set containing the origin in its interior, and let us write, for $f\in C_b(\RR^n)$, 
$$ \Phi(K,f) = \int f(x)\, d\Phi(K,x).$$
Consider the following three topologies on $\Curv(\RR^n)$.

\begin{enuma}
	\item The topology of weak convergence, defined by the family of seminorms
$$ p_{B,f}(\Phi) = \sup_{K\subset B} |\Phi(K,f)|, \quad f\in C_0(\RR^n).$$
\item The topology of compact convergence, defined by the family of seminorms
	\begin{equation}\label{eq:def-cc} p_{B,\calF}(\Phi) = \sup_{K\subset B,\ f\in \calF } |\Phi(K,f)|,\end{equation}
	for all compact subsets $\calF\subset C_0(\RR^n)$. 
\item The topology of bounded convergence, defined by the norm
$$ \| \Phi\|_B = \sup_{K\subset B} \|\Phi(K)\| .$$

\end{enuma}

For only one of these topologies  the natural action of the general linear group is continuous. From this perspective, the topology  of compact convergence is the most natural topology on $\Curv(\RR^n)$, and we will henceforth  exclusively work with this topology.

\begin{definition}
	Unless explicitly stated otherwise, we equip $\Curv(\RR^n)$ with the topology of compact convergence.
\end{definition}

In the rest of this section, we will establish the completeness  of the topologies of compact and bounded convergence. We start  by reminding the reader that the Riesz representation theorem asserts that for every locally compact  Hausdorff space $X$, the map $\mu\mapsto [ f\mapsto \int f \, d\mu]$ is an isometric isomorphism from $M(X)$ to $C_0(X)'$. Here $C_0(X)$ denotes the closure of $C_c(X)$, the space of compactly supported continuous functions with respect to the sup norm $\|f\|= \sup_{x\in X}|f(x)|$.  In particular,  the total variation norm of $\mu\in M(X)$ can be expressed as 
$$ \|\mu\| = \sup\left \{ \int_X f \, d\mu\colon f\in C_0(X),\  \|f\|\leq 1\right\}.$$

In combination with the principle of uniform boundedness, this description of the total variation norm implies the following well-known fact.

\begin{lemma}\label{lemma:uniform} Let  $\calM \subset M(X)$. If  for every $f\in C_0(X)$  the set 
	$$ \left \{ \int_X f \, d\mu\colon \mu\in \calM\right\}$$
	is bounded, then $\sup \{ \|\mu\| \colon \mu \in\calM\} <\infty$. 
\end{lemma}
\begin{proof}
	This follows from an application of the principle of uniform boundedness (see, e.g., \cite[5.12]{Folland:RealAnalysis}) to the family $I_\mu(f)= \int f\, d \mu$, $\mu\in \calM$, of linear functionals.
\end{proof}

While the weak topology on $E'$ is never complete for infinite dimensional metrizable locally convex spaces \cite[p. 148]{SchaeferWolff}, there are many spaces for which the topology of compact convergence is complete.

\begin{lemma} \label{lemma:complete-compact} Let $E$ be a Fr\'echet space. Then the topology of compact convergence on $E'$ is complete.
\end{lemma}
\begin{proof} Since the range of any null sequence in $E$ is relatively compact, the  statement follows from \cite[Remark IV.6.1]{SchaeferWolff}.
\end{proof}

\subsection{The topology of compact convergence}

\begin{lemma}\label{lemma:finite} Let $B\subset \RR^n$ be a bounded subset containing the origin in its interior.	Then the seminorms \eqref{eq:def-cc}
	define a Hausdorff topological vector space topology on  $\Curv(\RR^n)$. 
\end{lemma}
\begin{proof}
	Since $B$ is relatively compact, Lemma~\ref{lemma:uniform} implies that, for every curvature measure, $\|\Phi\|_B$ is finite. 
	Since compact subsets are bounded, for every compact set  $\calF\subset C_0(\RR^n)$, there exists $C>0$ such that  $p_{B,\calF}(\Phi)\leq C \|\Phi\|_B$. Hence  $p_{B, \calF}(\Phi)$ is finite.

Suppose that the curvature measure $\Phi$ satisfies $p_{B,\calF}(\Phi) =0$ for all compact subsets $\calF\subset C_0(\RR^n)$.  Given a convex body $K$, there is $\lambda>0$ such that a translate of $\lambda K$ is contained in $B$. By the translation invariance of $\Phi$, it follows that $\Phi(\lambda K)=0$. Since  this holds for all sufficiently small $\lambda>0$, the Kiderlen--Weil decomposition (Theorem~\ref{thm:KW}) implies 
$ \Phi(K)=0$. 
We conclude that $
\Phi=0$. This shows that the seminorms $p_{B, \calF}$ induce a Hausdorff topology.	
\end{proof}

\begin{lemma}\label{lemma:proj} Let $\Phi= \Phi_0+ \cdots + \Phi_n$ denote the Kiderlen--Weil decomposition.
	 Then 
	$\Phi\mapsto \Phi_i$ is continuous in the topology of compact convergence.  
\end{lemma} 
\begin{proof}

For every positive number $\lambda$ we define an operator $T_\lambda \colon \Curv(\RR^n)\to \Curv(\RR^n)$ by 
$$ (T_\lambda \Phi )(K,\beta) = \Phi(\lambda K, \lambda \beta).$$ 
From the definition of the norm it is evident that 
$$ p_{B, f} (T_\lambda \Phi) = p_{\lambda B,  f\circ \lambda^{-1}} (\Phi).$$

In terms of the operator $T_\lambda$, the Kiderlen--Weil decomposition can be restated as 
$$ T_\lambda \Phi = \Phi_0 + \lambda \Phi_1 +\cdots +\lambda^n \Phi_n, \quad \lambda >0.$$
This relation can be inverted, since it can be formulated in terms of a Vandermonde matrix.  Let $0<\epsilon$  be so small that $r B\subset B$ for all $0<r\leq \epsilon$ and choose numbers $0<\epsilon_0<\cdots <\epsilon_n\leq \epsilon$. Then exist numbers $a_{ik}$ such that 
$$ \Phi_i = \sum_{k=0}^n a_{ik} T_{\epsilon_k}  \Phi.$$

The set
$$ \calF' := \{  f\circ\epsilon_k^{-1}\colon  f\in \calF, \ k=0,\ldots, n\}$$
is uniformly bounded and equicontinuous, hence again compact by Arzel\`a--Ascoli. 
Using $\epsilon_k B\subset B$, we conclude that  
$$ p_{B,\calF} (\Phi_i)\leq  \max_{k} |a_{ik}| \sum_{k=0}^n p_{\epsilon_k B,\calF' } ( \Phi)\leq  
C  p_{B, \calF'}(\Phi).$$
\end{proof}

\begin{lemma}\label{lemma:equivalent-compact}
	For any two bounded subsets $B,B'\subset \RR^n$ containing the origin in their interior, the seminorms $\{ p_{B,\calF}\colon \calF \text{ compact} \} $ and $\{  p_{B',\calF}\colon \calF \text{ compact}\} $ are equivalent.
\end{lemma} 
\begin{proof} Choose $\lambda >0$ such that $B'\subset \lambda B$ and put $\calF'=  \{  f\circ \lambda \colon  f\in \calF\}$. Then, using Lemma~\ref{lemma:proj}, we obtain
	$$ p_{B',\calF} (\Phi)\leq p_{\lambda B,\calF} (\Phi)  \leq \sum_{i=0}^n p_{\lambda B,\calF} (\Phi_i) 
	= \sum_{i=0}^n  \lambda^i p_{ B,\calF' }(\Phi_i) \leq C  p_{B,\calF''}( \Phi ).$$ 	
\end{proof}

\begin{proposition}\label{prop:compact-complete}
	The topology of compact convergence on $\Curv(\RR^n)$ is complete.
\end{proposition}
\begin{proof}
		Let $(\Phi^{(s)})_{s\in \calI}$ be a Cauchy net in $\Curv(\RR^n)$ for the topology of compact convergence. Equivalently, for every  compact subset $\calF\subset   C_0(\RR^n)$ and every $\epsilon>0$, there exists $s_0\in \calI$ such that $s,t\geq s_0$ implies, 
	\begin{equation}\label{eq:net-conv}
		p_{B,\calF}(\Phi^{(s)} - \Phi^{(t)})<\epsilon. \end{equation}
	By Lemma~\ref{lemma:equivalent-compact}, the same is true for the  seminorms $ p_{rB,\calF}$ with $r>0$. 
	It follows that for every convex body $K$ in $\RR^n$, the measures $(\Phi^{(s)}(K))_{s\in \calI}$ form a Cauchy net  in $M(\RR^n)$ for the topology of compact convergence. Since this topology is complete (Lemma~\ref{lemma:complete-compact}), there exists a complex Borel measure $\Phi(K)$ such that 
	$\lim_{s\in \calI} \Phi^{(s)} (K)= \Phi(K)$.
	
	To prove locality,  let $K,L$ be two convex bodies in $\RR^n$ and suppose that $K\cap U = L\cap U$ for some open set $U$. Since $\{f\}$ is compact, we have $\Phi^{(s)}(K,f)\to \Phi(K,f)$ for every $f\in C_0(\RR^n)$. Consequently, 
	$$ \Phi(K,f)-\Phi(L,f)= \lim_{s\in \calI} \Phi^{(s)}(K,f)-\Phi^{(s)}(L,f) =0$$ 
	for every continuous, compactly supported  function $f$ in $U$. It follows that 
	$\Phi(K,\beta)-\Phi(L,\beta)=0$ for every Borel set $\beta \subset U$. 
	
	Similarly, we can conclude that $\Phi(K+x,\beta+x)= \Phi(K,\beta)$ holds for all $x\in \RR^n$, Borel sets $\beta$ and convex bodies $K$. 
	
	To prove the continuity of $\Phi$, let
	$(K_i)$ be a sequence of convex bodies converging to $K$.
	Choose $r>0$ so large that $K$ and all the $K_i$ are contained in $rB$.
	Passing to the limit in $t$ in  \eqref{eq:net-conv} implies that   for every  $f\in C_0(\RR^n)$ and $\epsilon>0$ there exists a number  $s_0$ such that 
	\begin{equation}\label{eq:weak_uniform_bound}\sup_{K\subset rB} |\Phi^{(s)}(K,f) - \Phi(K,f)| \leq \epsilon \end{equation}
	for all $s\geq s_0$.
	Fix $f\in C_0(\RR^n)$ and choose $s$ so large that \eqref{eq:weak_uniform_bound} holds. Choose $i_0$ so that 
	$$ |\Phi^{(s)}(K_i,f) - \Phi^{(s)}(K,f)|<\epsilon$$
	for all $i\geq i_0$.
	
	We conclude that 
	\begin{align*}|\Phi(K_i,f)& - \Phi(K,f)|\\  & \leq  |\Phi(K_i,f)-\Phi^{(s)}(K_i,f) | 
		+ |\Phi^{(s)} (K_i,f)-\Phi^{(s)}(K,f) | + |\Phi^{(s)} (K,f)-\Phi(K,f)|\\
		&\leq 3 \epsilon,
	\end{align*} 
	for all $i\geq i_0$ and hence $ \Phi(K_i)\stackrel{w}{
		\to} \Phi(K)$, as desired.
	
	 Passing to the limit  in  \eqref{eq:net-conv} implies that   for every $\epsilon>0$ there exists a number  $s_0$ such that 
	$ p_{B,\calF}( \Phi^{(s)} -\Phi)\leq \epsilon $ for all $s\geq s_0$. We conclude that $\lim_{s\in \calI}\Phi^{(s)} = \Phi$ in the topology of compact convergence.
\end{proof}

\subsection{The norm topology}

In the proof of Lemma~\ref{lemma:finite} we have already shown that $\|\Phi\|_B$ is always finite and that $\|\Phi\|_B=0$ implies $\Phi=0$. As above, one can prove  that $\Phi\mapsto \Phi_i$ is continuous in the norm topology and that, for any two bounded subsets $B,B'\subset \RR^n$ containing the origin in their interior, the norms $\|\cdot\|_B$ and $\|\cdot\|_{B'}$ are equivalent.

\begin{proposition}
	The norm $\|\cdot\|_B$ on  $\Curv(\RR^n)$ is complete.
\end{proposition}
\begin{proof}
Suppose that $(\Phi^{(s)})$ is a Cauchy sequence in $\Curv(\RR^n)$ for the norm $\|\cdot \|_B$. Then $(\Phi^{(s)})$ is also a Cauchy sequence for the topology of compact convergence and Proposition~\ref{prop:compact-complete} implies the existence of $\Phi\in \Curv(\RR^n)$ such that $\lim_{t\to\infty} \Phi^{(t)} =\Phi$ in the topology of weak convergence. By the Riesz representation theorem, for every $\epsilon>0$, there exists $s_0\in \NN$ such that  $s,t\geq s_0$ implies, for all $K\subset B$ and all $f$ with $\|f\|\leq 1$, 
$$ |\Phi^{(s)}(K,f)-\Phi^{(t)}(K,f)|<\epsilon.$$
Passing to the limit in $t$ in this inequality, yields $\|\Phi^{(s)} - \Phi\|_B\leq \epsilon$ for all $s\geq s_0$. This finishes the proof. 
\end{proof}

\section{The kernel of the globalization map}

\label{sec:kernel}

In this section, we connect curvature measures with valuations via two fundamental maps. 

\subsection{The globalization map and Bernig embedding}
According to Weil's theorem (Theorem~\ref{thm:weil}), the following  map is well defined. 

\begin{definition}
	The map $\glob \colon \Curv(\RR^n)\to \Val(\RR^n)$, defined by 
	$$ \glob(\Phi) = \Phi(K,\RR^n)= \Phi(K,1)$$
	is called the globalization map. 
\end{definition}

The following properties of the globalization map are evident from its definition.
\begin{lemma}
	The globalization map is 
	\begin{enuma}
		\item linear,
		\item continuous, and
		\item  preserves degree and parity:
		$$\glob(\Curv_i^\pm)\subset \Val_i^{\pm}.$$ 
	\end{enuma}
\end{lemma}

We will see in the next section that when restricted to smooth curvature measures and valuations, the globalization map becomes a surjection. Hence the problem of understanding the space of all curvature measures is essentially reduced to understanding  the kernel of the globalization map. We introduce the following notation.
\begin{definition}
	The kernel of the globalization map is denoted by $\Z(\RR^n)$. The subspaces of curvature measures of fixed degree and parity are denoted by $\Z_i^-(\RR^n)$ and $\Z_i^+(\RR^n)$.
\end{definition}

The crucial tool in our investigation of the kernel of the globalization will be an embedding of $\Z(\RR^n)$ into another space of valuations that we  describe next. 
If $V$ is a complex vector space, we denote by $\Val(\RR^n,V)$ the vector space of continuous translation invariant valuations with values in $V$.  We call $\phi\in \Val(\RR^n, V)$ even if $-\phi(-K)= \phi(K)$, and odd if  $\phi(-K)= \phi(K)$, for all convex bodies $K$. Defined in this way, the even elements of $\Val(\RR^n,\CC^n)$ are precisely those which are invariant under $-\id$ for the natural action of $GL(n,\RR)$ on $\Val(\RR^n,\CC^n)$. As a consequence, the Bernig embedding will be parity preserving, see Lemma~\ref{lemma:Bernig} below.

\begin{lemma}
	Let $\Phi\in \Z(\RR^n)$ belong to the kernel of the globalization map. Define 
	$$ B(\Phi)\colon \calK(\RR^n)\to \CC^n$$
	by $$ B(\Phi)(K) = ( \Phi(K,x_1), \ldots,  \Phi(K,x_n)),$$
	where $x_i\colon \RR^n\to \RR$ are the coordinate functions. 
	Then $B(\Phi)\in \Val(\RR^n,\CC^n)$. 
\end{lemma}
\begin{proof}
	Since continuity and the valuation property are clearly satisfied, it only remains to consider  invariance under translations. For every $v\in \RR^n$, we have  
	$$ \Phi(K+v,x_i)= \Phi(K, x_i + v_i)= \Phi(K,x_i),$$
	since $\Phi$ globalizes to zero. This shows that $B(\Phi)(K+v)= B(\Phi)(K)$ as desired.
\end{proof}

Recall that a continuous function $g\colon \RR^n\to \RR$ is called piecewise linear if there exists a locally finite simplicial subdivision $\calT$ of $\RR^n$ such that, for each $T\in \calT$, the restriction $g|_T$ is affine. 

\begin{lemma} \label{lemma:pl}
	Let $f\in C(\RR^n)$, let $K\subset \RR^n$ be compact, and let $\epsilon>0$. Then there exists a piecewise linear function $g$ such that 
	$$ \sup_K| f-g |< \epsilon.$$
\end{lemma} 
\begin{proof}
	The statement of the lemma is well known. To obtain $g$, choose  a sufficiently fine triangulation $\calT$, let $g(x)=f(x)$ at the vertices, and define $g|_T$ by affine extension.
\end{proof}

The following injectivity property is the main reason for considering $B(\Phi)$.
In a slightly different context \cite[Lemma~4.5]{SolanesW:Spheres}, the proof of the following theorem was suggested by A.~Bernig.

\begin{theorem}
	Let $\Phi\in \Z(\RR^n)$ belong to the kernel of the globalization map. If $B(\Phi)=0$, then $\Phi=0$. 	
\end{theorem}
\begin{proof}
	Observe that if $\Phi$ satisfies $B(\Phi)=0$, then 
	$\Phi(K,h)=0$ for every affine function $h$. 
	
	Let $K$  be a fixed convex body and let $f\in C_b(\RR^n)$. Let $\epsilon>0$ and choose a piecewise linear approximation $g$ of $f$ as in Lemma~\ref{lemma:pl}. Let  $\wt{g|_T}\colon \RR^n\to \RR$ denote the affine function extending $g|_T$. Let us write $\{T_1,\ldots, T_m\}= \{ T\in \calT \colon K\cap T\neq \emptyset\}$. By the inclusion-exclusion principle for continuous valuations, see \cite[Theorem 5.1.1]{KlainRota}, and Lemma~\ref{lemma:supp}, 
	\begin{align*} \Phi(K, g) & = \sum_{k= 1}^m (-1)^{k+1} \sum_{1\leq i_1<\cdots<i_k\leq m}  \Phi(K\cap T_{i_1}\cap \cdots \cap T_{i_k} ,g)\\
		 & =
		\sum_{k= 1}^m (-1)^{k+1} \sum_{1\leq i_1<\cdots<i_k\leq m} \Phi(K\cap T_I,\wt{g|_{T_I}}) =0,\end{align*}
	where $T_I= T_{i_1}\cap \cdots \cap T_{i_k}$.
	Consequently, $$|\Phi(K,f)| \leq |\Phi(K,f-g)|+ |\Phi(K,g)|\leq \|\Phi(K)\| \epsilon.$$
	Since $\epsilon>0$ was arbitrary, we conclude that $\Phi(K)=0$. 
	
\end{proof}

\begin{definition} The map $B\colon \Z(\RR^n)\to \Val(\RR^n,\CC^n)$ is called the Bernig embedding.
\end{definition}

We end this subsection by collecting the basic properties of the Bernig embedding.

\begin{lemma} \label{lemma:Bernig} The Bernig embedding is 
	\begin{enuma}
		\item linear, 
		\item continuous, and 
		\item increases the degree by $1$ and preserves parity:
		$$ B(\Z_i^\pm) \subset \Val_{i+1}^{\pm}(\RR^n,\CC^n).$$
	\end{enuma}
\end{lemma}
\begin{proof}
	These properties are immediate consequences of the definition of $B$.
\end{proof}

\subsection{The image of the Bernig embedding}

In what follows it will be crucial to obtain precise information about the image of the Bernig embedding. This subsection contains two preliminary results in this direction. The first is a simple geometric statement, the second is slightly more involved. 

For an affine subspace $E\subset \RR^n$, let $\dir(E)= E-x$, for any $x\in E$, denote the direction space of $E$.

\begin{lemma}\label{lemma:subspace}
	If $\phi\in \Val_{k+1}(\RR^n,\CC^n)$ belongs to the image of the Bernig embedding, then 
	$$ \phi(K)\in \dir(\aff K). $$
\end{lemma}
\begin{proof} Suppose $\phi= B(\Phi)$. 
	Since $\phi$ is translation invariant, we may assume $0\in K$. Moreover, by the continuity of $\phi$, we may further assume that $K$ is a polytope. In this case, by Theorem~\ref{thm:polytopes}, we evidently have
	$$\phi(K)= B(\Phi)(K)= \sum_{F\in \calF_{k}(K)} \varphi(N(F,K)) \int_{F} x\, dx \in \spn K.$$  
\end{proof}

Suppose $\Phi\in \Z_{n-2}(\RR^n)$. Since  $B(\Phi)$ is an $(n-1)$-homogeneous valuation, McMullen's theorem (Theorem~\ref{thm:McMullen-n-1})  implies the existence of a continuous function $f\colon S^{n-1}\to \CC^n$ such that 
\begin{equation} \label{eq:B-f}   B(\Phi)(K)= \int_{S^{n-1}} f(u)\, d\!\S_{n-1}(K,u).
\end{equation} 
Recall that   the function $f$ is unique up to the addition of linear maps $\lambda\colon \RR^n\to \CC^n$. Observe  that for every such $\lambda$, the function $x\mapsto 
\langle \lambda(x),x\rangle$ is a homogeneous quadratic polynomial  on $\RR^n$.  Consequently,   condition \eqref{eq:cond-quadratic} below is independent of the particular choice of $f$.

\begin{proposition}\label{prop:imB}
		Suppose the continuous function $f\colon S^{n-1}\to \CC^n$ satisfies \eqref{eq:B-f}. Then there  exists a homogeneous quadratic polynomial $p\colon \RR^n\to \CC $ such that 
	\begin{equation}\label{eq:cond-quadratic} \langle f(u),u\rangle = p(u), \quad u\in S^{n-1}.\end{equation}
\end{proposition}
\begin{proof}	
	Let $S(u)$ be the simplex in $\RR^n$ bounded by the hyperplanes $x_i=0$ and $\langle u,x\rangle =1$ for a direction  $u\in S^{n-1}$ belonging to the positive orthant.
	
	By Theorem~\ref{thm:polytopes},  there exists a simple valuation $\varphi$ on $\calC_{2}$ such that 
	$$ \Phi(S(u), \beta) = \sum_{ i<j}  \varphi(N_{ij}) \vol_{n-2}(F_{ij}\cap \beta), 
	$$
	where $F_{ij}= F_i\cap F_j$ denotes the  codimension $2$ faces of $S(u)$ and the facets $F_0,\ldots, F_n$ are labeled such that $-e_i$ is the outer unit normal to $F_i$. 
	Since all faces of $S(u)$ are  simplices, their centroids are expressible as a sum of their vertices. Hence
	\begin{align*}    B(\Phi)(S(u))  = & \frac{1}{n-1} \sum_{1\leq  i<j}  \varphi(N_{ij})  \vol_{n-2}(F_{ij})  \sum_{k\neq i,j} u^{-1}_ke_k\\ &  + \frac{1}{n-1} \sum_{j=1}^n  \varphi(N_{0j}) \vol_{n-2}(F_{0j}) \sum_{k\neq j} u^{-1}_ke_k. 
	\end{align*}
	Since $\glob(\Phi)=0$, it follows that 
	\begin{align*}   \langle B(\Phi)(S(u)), u\rangle  &  = \frac{n-2}{n-1}  \sum_{ 1\leq i<j}  \varphi(N_{ij})  \vol_{n-2}(F_{ij}) +  \sum_{j=1}^n  \varphi(N_{0j}) \vol_{n-2}(F_{0j}) \\
		&= - \frac{1}{n-1}  \sum_{1\leq  i<j}  \varphi(N_{ij})  \vol_{n-2}(F_{ij}) \\
		&= 	- \frac{1}{(n-1)!} \frac{1}{u_1\cdots u_n}  \sum_{1 \leq  i<j}  \varphi(N_{ij})  u_i u_j, 
	\end{align*} 
	where we used
	$$ \vol_{n-2}(F_{ij}) = \frac{1}{(n-2)!} \frac{u_iu_j}{u_1\cdots u_n}.$$	
	
	On the other hand, if   $f\colon S^{n-1}\to \CC^n$ satisfies \eqref{eq:B-f}, then 
	$$(n-1)!  u_1\cdots u_n \langle  B(\Phi)(S(u)), u\rangle  =  \sum_{i=1}^n  \langle  f(-e_i),u\rangle u_i  +  \langle  f(u), u\rangle.$$
	We conclude  that there exists a homogeneous quadratic polynomial $p\colon \RR^n\to \CC$ such that, for $u\in S^{n-1}\cap (0,\infty)^n$, 
	$$ \langle f(u),u\rangle = p(u).$$
	
		Observe that the preceding argument is independent of the choice of orthonormal basis of $\RR^n$. Thus for each positive orthant $O$ there exists a quadratic polynomial $p_O$ such that $\langle f(u),u\rangle =p_O(u)$ for $u \in S^{n-1}\cap \interior O$.  Since a function on $\RR^n\setminus \{0\}$ that is locally a polynomial is globally a polynomial, it follows that $ \langle f(u),u\rangle  = p(u)$ for all $u\in S^{n-1}$. 
\end{proof}

An easy consequence of Lemma~\ref{lemma:subspace} and Proposition~\ref{prop:imB} is the following.
\begin{corollary}\label{cor:imB}
	For $k\leq n-2$, the image of the Bernig embedding  $ B(\Z_k^\pm)\subset \Val^\pm_{k+1}$ is not dense.
\end{corollary}

\subsection{Curvature measures of degrees $n-1$ and $n$}

\label{sec:degree}

As an application of the notions introduced in this section, we obtain an explicit description of the curvature measures in degrees 
$n-1$ and $n$, with succinct proofs. The degree 
$n$ case was already known to Kiderlen and Weil \cite{KiderlenWeil:Measure}, while the degree $n-1$ case seems to be new.

\begin{theorem}
	The vector space $\Curv_n(\RR^n)$ is  spanned by 
	$$ \C_n(K,\beta)= n \vol(K\cap \beta).$$
\end{theorem}
\begin{proof}
Let $\Phi\in \Curv_n(\RR^n)$. By a theorem of Hadwiger (Theorem~\ref{thm:hadwiger}), $\Val_n(\RR^n)$ is  spanned by the Lebesgue measure. Consequently, there exists a number $\alpha\in \CC$ such that 
$\Phi - \alpha \C_n$ globalizes to zero.  Since the degree of a continuous translation invariant valuation cannot exceed the dimension of the space, $\Val_{n+1}(\RR^n,\CC^n)=\{0\}$. Therefore $B(\Phi-\alpha \C_n )=0$ and hence $\Phi= \alpha \C_n$ by the injectivity of $B$. 
\end{proof} 

\begin{lemma}\label{lemma:div} For every convex body $K$, 
	$$ \int_{\RR^n\times S^{n-1}} x_iu_j \, d\Theta_{n-1}(K,(x,u))=\delta_{ij}\vol(K).$$
\end{lemma}
\begin{proof}By continuity it suffices to prove the claim for $K$ with $C^1$ boundary. For such $K$, by \cite[Lemma 4.2.3]{Schneider:BM}, 
	$$ \int_{S\RR^n} f(x,u)\, d\Theta_{n-1}(K,(x,u)) = \int_{\partial K}  f(x,n_K(x))\,  dx,$$
	where $n_K(x)$ denotes the outer unit normal vector at $x$. 
	Since $\operatorname{div}(x_i e_j)= \delta_{ij}$, the divergence theorem implies 
	$$\int_{S\RR^n} x_iu_j \, d\Theta_{n-1}(K,(x,u))= \int_{\partial K} \langle n_K(x), x_ie_j\rangle\,  dx= \delta_{ij} \vol(K).$$
\end{proof}

\begin{theorem} \label{thm:deg-n-1} Let $\Phi\in \Curv_{n-1}(\RR^n)$. Then there exists a unique continuous function $f\colon S^{n-1}\to \CC$ such that 
	$$ \Phi(K,\beta)= \int_{\pi^{-1}(\beta)} f(u)\, d \Theta_{n-1}(K,(x,u)).$$
\end{theorem} 
\begin{proof}
	By Theorem~\ref{thm:McMullen-n-1}, there exists a continuous function $g\colon S^{n-1}\to \CC$, unique up to the addition of linear functionals,  such that
	 $$ \glob( \Phi)(K)= \int_{S^{n-1}} g(u)\, d\!\S_{n-1}(K,u).$$
	  If we define a curvature measure $\Psi_g$ by 
	 $$ \Psi_g(K,\beta)=\int_{\pi^{-1}(\beta)} g(u)\, d\Theta_{n-1}(K,(x,u)),$$
	 then $\Phi-\Psi_g$ globalizes to zero. Moreover, adding a linear functional to $g$ will not change this.

	 By Theorem~\ref{thm:hadwiger}, there exists $\alpha=(\alpha_1,\ldots, \alpha_n)\in \CC^n$ such that 
	 $$ B(\Phi-\Psi_g)(K)= \vol(K)\alpha.$$
	  Define $f(u)= g(u)+ \alpha_1 u_1+ \cdots + \alpha_n u_n$. 
	 Applying Lemma~\ref{lemma:div} shows that 
	 $ B(\Phi-\Psi_f)= 0$ and hence $\Phi= \Psi_f$ by the injectivity of $B$. 
\end{proof}

\subsection{Simple curvature measures}

We call a curvature measure $\Phi\in \Curv(\RR^n)$ simple if $\Phi(K)=0$ whenever $K$ has empty interior. The following result is analogous to the Klain--Schneider characterization of simple valuations (Theorem~\ref{thm:klain-schneider}). For smooth curvature measures, it was established by Solanes and the second-named author in \cite{SolanesW:Spheres}. 

\begin{theorem}\label{thm:simple}
A curvature measure $\Phi \in \Curv(\RR^n)$ is simple if and only if $$\Phi\in \Z^+_{n-2}\oplus \Curv_{n-1}^-\oplus \Curv_n.$$
\end{theorem}
\begin{proof}
Without loss of generality we may assume that $\Phi$ is homogeneous of degree $i$. If $i=n$, there is nothing to prove. If $i=n-1$, then Theorem~\ref{thm:deg-n-1} implies the claim. We assume therefore $i\leq n-2$ from now on. 

Suppose that $\Phi$ is simple. Observe that $\glob(\Phi)\in \Val_i(\RR^n)$ is  simple as well. By the Klain--Schneider characterization of simple valuations, $\glob(\Phi)=0$.
Thus $\Phi\in \Z_i(\RR^n)$ and $B(\Phi)$ is a simple valuation. Applying   the Klain--Schneider characterization once again, we obtain $\Phi\in \Z_{n-2}^+$. 

Conversely, suppose that $\Phi\in \Z_{n-2}^+$. Let $H$ be a linear hyperplane in $\RR^n$ and let $\Phi|_H\in \Curv_{n-2}(H)$ denote the restriction of $\Phi$ to convex bodies in $H$. Observe that $\Phi|_H\in \Z_{n-2}^+(H)$. On the other hand,  Theorem~\ref{thm:deg-n-1} and Theorem~\ref{thm:McMullen-n-1} imply $\Z_{n-2}^+(H)=\{0\}$. Thus $\Phi|_H=0$ for every linear hyperplane $H$ and translation invariance implies that $\Phi$ is simple. 
\end{proof}

\section{Smooth curvature measures}
\label{sec:smooth}

 In this section, we introduce several key concepts from representation theory and we use them to establish the  logical implications between the main theorems and conjectures presented in the introduction. The reason for our choice of topology on $\Curv(\RR^n)$ will become apparent in Lemma~\ref{lemma:cont-G-action}. Since methods from representation theory presuppose a continuous  representation, we are forced to work with the topology of compact convergence.

\subsection{Preliminaries from infinite-dimensional representation theory}

We recall here some definitions and foundational  results from representation theory. For more information we refer the reader to the paper \cite{Casselman:canonical} and the introductory textbook \cite{Knapp:semisimple}.
The results of this section, which we will apply only for $G=GL(n,\RR)$ and $K=O(n)$,  hold more generally for $G$ the group of $\RR$-rational points of a  Zariski connected, reductive algebraic group defined over $\RR$ and a maximal compact subgroup $K$ of $G$. Throughout  $E$ will be a complete locally convex Hausdorff topological vector space over $\CC$. 

 A linear representation of $G$ on $E$ is a group homomorphism $\pi\colon G\to GL(E)$, where $GL(E)$ denotes the group of invertible continuous endomorphisms of $E$. A representation is called continuous  if the map $G\times E\to E$, $(g,v) \mapsto \pi(g)v$ is continuous. The following lemma, an 	immediate consequence of  the principle of uniform boundedness, is useful for proving the continuity of a representation.
 
 \begin{lemma}\label{lemma:cont}
 	A linear representation $(\pi,E)$ is continuous if and only if the following properties hold:
 	\begin{enuma}
 		\item For every $v\in E$, the map $G\to E$, $g\mapsto \pi(g) v$,  is continuous; and  
 		\item For every compact subset $C\subset G$, the set of operators $\{ \pi(g) \colon g\in C\}$ is equicontinuous: for every neighborhood $V_1$ of zero in $E$ there exists another neighborhood $V_2$  of zero such that, 	for all $g\in C$, 
 		$$ \pi(g)( V_2)\subset V_1.$$
 	\end{enuma}	
 	
 \end{lemma}

If $(\sigma,F)$ is another continuous representation of $G$, a continuous linear map $T\colon E\to F$ satisfying $T\circ \pi(g)= \sigma(g) \circ T$ is called $G$ homomorphism. 
The representation $\pi$ is called irreducible if $\{0\}$ and $E$ are the only closed invariant subspaces of $E$.

A vector $v\in E$ is called smooth if the map  $G\to E$, $g\mapsto \pi(g)v$, is  $C^\infty$. The subspace of all smooth vectors is denoted by $E^\infty$, and it is invariant under $G$. The subspace $E^\infty$ embeds canonically into $C^\infty(G,E)$,  and it  inherits from this embedding a canonical topology. The representation of $G$ on $E^\infty$ is continuous with respect to this topology. If $E=E^\infty$ as topological vector spaces, then $E$ is said to be smooth.  The space of smooth vectors is also a representation of  the Lie algebra 
$\fgg$  of $G$ with the action defined by 
$$ \pi(X)v= \left.\frac{d}{dt}\right|_{t=0} \pi(\exp(tX)) v.$$

A vector $v\in E$ is called $K$-finite if it is contained in a finite-dimensional $K$ invariant subspace. The subspace of  $K$-finite vectors is denoted by $E_K$. It is the algebraic sum of the $K$ isotypic components of $E$. 

The intersection $E^\infty \cap E_K$ is dense in $E$ and invariant under $\fgg$ and $K$; it is called the $(\fgg,K)$ module associated to $(\pi,E)$.  If $T\colon E\to F$ is a $G$ homomorphism, then $T(E^\infty)\subset F^\infty$ and $T(E_K)\subset F_K$. In particular,  $T$ induces a homomorphism of the associated $(\fgg, K)$ modules.

\begin{lemma} \label{lemma:K-finite}Let $(\pi,E)$ be a continuous representation of  $G$. Suppose that every $K$-type in $E$ appears with finite multiplicity. Then 
\begin{enuma}
	\item $E_K\subset E^\infty$.
\item The correspondence which takes a $(\fgg,K)$ invariant subspace of $E_K$ to its closure in $E$ is a bijection between $(\fgg,K)$ invariant subspaces of $E_K$ and closed $G$  invariant subspaces of $E$.  The inverse map takes a closed $G$ invariant subspace of $E$ to its intersection with $E_K$. 
	
\end{enuma}
\end{lemma}

In general, the associated $(\fgg,K)$ module does not determine $(\pi,E)$ up to isomorphism. In fact, we will encounter this issue in our investigation of $\Curv(\RR^n)$ as a representation of $GL(n,\RR)$. 
A fundamental result of Casselman and Wallach provides a positive answer for a certain class of representations. A $(\fgg,K)$ module $V$ is said to have finite length if there exist $(\fgg,K)$ submodules 
\begin{equation}\label{eq:composition} V=V_0 \supsetneq V_1\supsetneq \cdots \supsetneq V_m=\{0\}\end{equation}
such that $V_i/V_{i+1}$ is irreducible. The number $m$ is called the length of $V$. A $(\fgg,K)$ module is called a Harish-Chandra module if it  has finite length and every $K$ isotypic component of $V$ is finite-dimensional. The result of Casselman and Wallach can now be stated as follows: If $V$ is a Harish-Chandra module then there exists a unique smooth Fr\'echet space representation $(\pi,E)$ of moderate growth such that $E_K$ is isomorphic to $V$. We refer the reader for further information to the original paper by Casselman \cite{Casselman:canonical} and, for a simplified proof, to \cite{BernsteinKrotz}.
We shall not define moderate growth here. For our purposes,  it suffices to know that if $E$ is a Banach space, then $E^\infty$ has moderate growth. 

One useful consequence of this fundamental result of Casselman and Wallach is that  a $G$ homomorphism between certain smooth representations always has closed image. 

\begin{theorem}[Casselman--Wallach] \label{thm:CW}
	Let $E$ and $F$ be smooth  Fr\'echet representations of $G$  of moderate growth. 
	Suppose that $F_K$ is a Harish-Chandra module. Then every $G$ homomorphism 
	$T\colon E\to F$ has closed image.
\end{theorem}

 A composition series of a continuous  representation $(\pi, E)$ of $G$ is a strict  chain of closed invariant subspaces $E= E_0  \supsetneq E_1 \supsetneq \cdots \supsetneq E_m =\{ 0\}$ such that $E_i/E_{i+1}$ is irreducible. By Lemma~\ref{lemma:K-finite}, if every $K$-type appears with finite multiplicity, then $m$ depends only on $E$ and is called the length of $E$. If a  $(\fgg,K)$ module $V$ has finite length, then the length of every strict chain of submodules is bounded by the length of $V$  and can be refined to a composition series of $V$. 

\medskip

We close with a folklore lemma.

\begin{lemma}\label{lemma:smooth-sec}
Let $H$ be a closed subgroup of $G$ and let $M= G/H$. Let $E\to M$ be a smooth  $G$ equivariant complex vector bundle over $M$.  Then the natural representation of $G$ on  $\Gamma(M,E)$,  the space of smooth sections of $E$, is smooth.
\end{lemma}
\begin{proof}
	We consider only the case $G=GL(n,\RR)$. Since $\Gamma(M,E)\subset C^\infty(G,V)$, where $V$ is a finite-dimensional complex vector space, it suffices to show that the  standard representation of $G$ on $C^\infty(G,V)$, $(\pi(g)f)(h)= f(g^{-1}h)$,  is smooth. For $f\in C^\infty(G,V)$, by Taylor expansion with integral form of the remainder, 
	\begin{align*} \|A-A_0\|^{-1} \Big( f(AB) &  -f(A_0B)-  \sum_{ij} (AB-A_0B)_{ij}\frac{\partial f}{\partial x_{ij}} (A_0B) \Big) \\
		&= \|A-A_0\|^{-1} \int_0^1 R(t) (1-t) \, dt\end{align*}
	with 
$$R(t)= \sum_{i,j,k,l}	(AB-A_0B)_{ij} (AB-A_0B)_{kl} \frac{\partial^2 f}{\partial x_{ij} \partial x_{kl} } ((1-t)A_0B + tAB).$$
	Consequently, for every compact subset $K\subset G$, 
	$$ \sup_{B\in K}  \|A-A_0\|^{-1} \Big( f(AB)   -f(A_0B)-  \sum_{i,j} (AB-A_0B)_{ij}\frac{\partial f}{\partial x_{ij}} (A_0B) \Big) \to 0$$
	as $\|A-A_0\|\to 0$.
It follows that $\varphi\colon G\to C(G,V)$, $A\mapsto \pi(A^{-1})f$, is continuously differentiable with first partial derivatives 
$$ \frac{\partial \varphi}{\partial x_{ij}}(A)=    [ B\mapsto \sum_k B_{jk} \frac{\partial f}{\partial x_{ik}} (AB)].$$
Repeating this argument shows that $\varphi$ is infinitely differentiable. We conclude that $C^\infty(G,V)$ is contained in the subspace of smooth vectors of $C(G,V)$.  Since the converse is trivially true, we have $C^\infty(G,V)= C(G,V)^\infty$. The topologies coincide by the open mapping theorem.
\end{proof}

\subsection{The action of the general linear group on curvature measures}

\begin{definition}
	Let $g\in GL(n,\RR)$ and let $\Phi\in \Curv(\RR^n)$. We define the curvature measure $g\cdot \Phi$ by
	$$ (g\cdot \Phi)(K,\beta)= \Phi(g^{-1}K, g^{-1}\beta).$$
	We let $\pi(g)\colon \Curv(\RR^n)\to \Curv(\RR^n)$ denote the corresponding linear map. 
\end{definition}
For every function $f\colon \RR^n\to \CC$ and $g\in GL(n,\RR)$ we define 
$$ g\cdot f(x)= f(g^{-1} x).$$ 
Observe that with this notation, 
$$ (g\cdot \Phi)(K,f)= \Phi(g^{-1} K, g^{-1} \cdot f).$$

\begin{lemma}\label{lemma:cont-G-action}The map $GL(n,\RR)\times \Curv(\RR^n)\to \Curv(\RR^n)$, $(g,\Phi)\mapsto g\cdot \Phi$, defines a continuous  representation of the general linear group on $\Curv(\RR^n)$.
\end{lemma}
\begin{proof}

	We verify  properties (a) and (b) of Lemma~\ref{lemma:cont}.
To prove (a), it suffices to show that for every 
 sequence $(g_i)$  in $GL(n,\RR)$ converging to the identity matrix  $I$,  the sequence $(g_i\cdot \Phi)$ converges to $\Phi$. 

Let $(g_i)$ be a sequence in $GL(n,\RR)$ converging to $I$. Then also $g^{-1}_i\to I$ and both $(g_i)$ and $(g^{-1}_i)$ are uniformly bounded in the operator norm $\|g\| = \sup_{x\in B^n} |g x|$. 
Consequently, by Lemma~\ref{lemma:uniform}, 
$$ M= \sup\{ \|\Phi(g^{-1}_i K)\|\colon K\subset B^n, \ i\in \NN\}$$
is finite. 
Moreover, since $\supp \Phi(K)\subset K$, it follows that there exists $r>0$ such that the ball $rB^n$ contains the supports of all  measures $\Phi(g^{-1}_i K)$ for all $i$ and all $K\subset B^n$. 

Observe that for every $K\subset B^n$, the Hausdorff distance satisfies
$$ d_H(g^{-1}_i K, K)\leq \|g^{-1}_i-I\|.$$

Let $\calF$ be a compact subset of $C_0(\RR^n)$ and  let $\epsilon>0$. By the compactness of $\calF$, there exists a finite set $\{f_1,\ldots, f_N\}$ such that 
$$ \calF\subset \bigcup_{j=1}^N B_{\epsilon/M}(f_j),$$
where $B_r (f)$ denotes the open $r$-ball with center $f$.  
Choose $i_0$ such that 
$$ |\Phi(g_i^{-1}K, f_j) - \Phi(K,f_j)|<\epsilon$$
for all $i\geq i_0$, all $K\subset B^n$, and all $j\in \{1,\ldots, N\}$. 

For every $K\subset B^n$ and $f\in \calF$, we  thus obtain, for $i\geq i_0$, 
\begin{align*}
	|(g_i \cdot \Phi)(K,f) - \Phi(K,f)|& \leq  |\Phi(g_i^{-1}  K,g_i^{-1}  \cdot f) - \Phi(g^{-1}_i K,f)| \\
	& \quad \quad + | \Phi(g^{-1}_i K, f) - \Phi(g^{-1}_i K,f_j)| + |\Phi(g_i^{-1}K, f_j) - \Phi(K,f_j)|\\
	&\leq  M \,\sup_{rB^n} | g_i^{-1} \cdot f - f|  + 2\epsilon.
\end{align*}
It follows that $  p_{B,\calF} (g_i\cdot \Phi-  \Phi)<3\epsilon $ for sufficiently large $i$.

Now we turn to (b). Without loss of generality we may assume $V_1=\left\{\Phi\colon  p_{B,\calF}(\Phi)<1\right\}$. 
Since $C$ is compact, the set
$$ \calF' := C^{-1} \cdot \calF = \{ g^{-1}\cdot f\colon g\in C, \ f\in \calF\}$$
is uniformly bounded and equicontinuous, hence again compact by Arzel\`a--Ascoli. 
Moreover, there exists a bounded subset $B'\subset \RR^n$, containing the origin in the interior, such that $g^{-1}K\subset B'$ for every $g\in C$ and $K\subset B$. Define 
$$ V_2 = \{ \Phi \colon p_{B',\calF'}(\Phi)<1\}.$$ Then, for every $\Phi\in V_2$, 
$$ p_{B,\calF}(g\cdot \Phi) =  \sup_{K\subset B, \ f\in \calF} | \Phi(g^{-1} K,g^{-1}\cdot f)| \leq p_{B', \calF'} (\Phi)<1$$ 
and hence $ g\cdot \Phi\in V_1$, as required. 
\end{proof}

For $\phi\in \Val(\RR^n,\CC^n)$ we define 
$$ (g\cdot \phi)(K) = g \phi(g^{-1} K).$$
\begin{lemma} Let $B\subset \RR^n$ be a bounded subset containing the origin in its interior. Then 
	$$ \|\phi \|_B = \sup_{K\subset B} |\phi(K)|$$
	defines a Banach space norm on $\Val(\RR^n,\CC^n)$. The representation of $GL(n,\RR)$ on $\Val(\RR^n,\CC^n)$ is continuous in this topology.
\end{lemma} 
\begin{proof}
	This follows in a straightforward way from the corresponding statement for $\Val(\RR^n)$. 
\end{proof}

\begin{lemma} The following properties hold:
	\begin{enuma}
		\item 	$\glob\colon \Curv(\RR^n)\to \Val(\RR^n)$ is a $GL(n,\RR)$ homomorphism. 
		\item $\Z(\RR^n)$ is a  closed invariant subspace of $\Curv(\RR^n)$. 
	\item  
	$ B\colon \Z(\RR^n)\to \Val(\RR^n,\CC^n)$ is a $GL(n,\RR)$ homomorphism. 
	\end{enuma} 
\end{lemma}
\begin{proof}
	(a) is evident from the definition and (b) is an immediate consequence of (a). 
	
	To prove (c), first note that $g^{-1}\cdot x_i=   \sum_{j=1}^n g_{ij} x_j$
	for coordinate functions 
 and $$g  \alpha = (\sum_j g_{1j} \alpha_j, \ldots, \sum_j g_{nj} \alpha_j)$$ for $\alpha= (\alpha_1,\ldots, \alpha_n)\in \CC^n$. Hence
	\begin{align*} B(g\cdot \Phi)(K)& = ( \Phi(g^{-1}K, g^{-1}\cdot x_1), \ldots, \Phi(g^{-1}K, g^{-1}\cdot x_n))\\
		& =  (\sum_j g_{1j} \Phi(g^{-1}K,  x_j), \ldots, \sum_j g_{nj}\Phi(g^{-1}K, x_j))\\
		& = g B(\Phi)(g^{-1} K)= (g\cdot B(\Phi))(K).
	\end{align*}
\end{proof}

\begin{lemma}\label{lemma:CW-thm}
	$\Val(\RR^n,\CC^n)^\infty$ is a Fr\'echet space representation of moderate growth and its associated $(\fgg,K)$ module is Harish-Chandra.
\end{lemma}
\begin{proof}
Since $\Val(\RR^n,\CC^n)$ is a Banach space, it follows that $\Val(\RR^n,\CC^n)^\infty$ is a Fr\'echet space and has moderate growth. 

Observe that $\Val(\RR^n,\CC^n)$ and $\Val(\RR^n)\otimes \CC^n$ are isomorphic as representations of  $G=GL(n,\RR)$.  
By  \cite{Alesker:Irreducibility}, each $K$-type appears in $\Val(\RR^n)$ with finite multiplicity. Hence the Pieri rule for $O(n)$, which shows how the tensor products decompose into irreducible representations, implies that also the $K$-types in $\Val(\RR^n,\CC^n)$ appear with finite multiplicity. 

Let $E=\Val(\RR^n)$ and let  $E_K$ denote its $(\fgg, K)$ module. Then $E_K\otimes \CC^n$ is the $(\fgg, K)$ module of $\Val(\RR^n,\CC^n)$, and it remains to show that it has finite length. 
  We are going to use the fact that a $(\fgg,K)$ module has finite length if and only if it is finitely generated over $U(\fgg)$, the universal enveloping algebra of $\fgg$, see \cite[p. 394]{Casselman:canonical}. 
  
It follows from Alesker's irreducibility theorem (Theorem~\ref{thm:alesker}) that $E_K$ has finite length and is hence finitely generated over $U(\fgg)$. If   $v_1,\ldots, v_m$ are the generators of $E_K$, then one immediately verifies that  $\{ v_i\otimes e_j\}$ is a finite set of generators for $E_K\otimes \CC^n$. 
\end{proof}

Now that we have introduced the necessary background from representation theory, we can easily deduce Theorem~\ref{thm:conj}  from Theorem~\ref{thm:Val2}, the proof of which will occupy the next section.

\begin{proposition} Let $k\in \{0,\ldots, n-2\}$. If $\Val_{k+1}^\pm(\RR^n,\CC^n)$ has length $2$, then $\Z^\pm_k$ is irreducible. \end{proposition}
\begin{proof}
	 Let $W$ denote the closure of $ B( \Z_k^\pm)$ in $\Val^\pm_{k+1}(\RR^n,\CC^n)$. Then $W$ is a $GL(n,\RR)$ invariant and proper (by Corollary~\ref{cor:imB}) subspace of $\Val^\pm_{k+1}(\RR^n,\CC^n)$. If the latter  representation has  length $2$, then $W$ is irreducible. Since $W$ and $\Z_k^\pm$ have isomorphic associated $(\fgg, K)$ modules, we conclude that $\Z_k^\pm$ is irreducible.
\end{proof}

\begin{definition}
A  curvature measure $\Phi\in \Curv(\RR^n)$ is called smooth if there exists a number $c \in \CC$  and a smooth differential form  $\omega\in \Omega^{n-1}(S\RR^n)^{tr}$  such that
	$$\Phi(K,\beta)=  c \vol(K\cap \beta ) + \int_{\nc(K)} \mathbf{1}_{\pi^{-1}(\beta)}\, \omega$$
	for every convex body $K$ and Borel set $\beta$. The subspace of smooth curvature measures is denoted by $\Curv^{sm}(\RR^n)$. 
\end{definition}

The action of the general linear group gives rise to another notion of smoothness for curvature measures. 

\begin{definition}
	A curvature measure $\Phi\in \Curv(\RR^n)$ is called $GL(n,\RR)$ smooth if it is a smooth vector for the representation of $GL(n,\RR)$ on $\Curv(\RR^n)$. 
	The subspace of these curvature measures is denoted by $\Curv^\infty(\RR^n)$.
\end{definition} 

For translation invariant continuous valuations,  Alesker \cite{Alesker:HLComplex} defined $\Val^\infty$ as the subspace of smooth vectors of the Banach space representation $\Val(\RR^n)$ of $GL(n,\RR)$.  Recall that $\Val^{sm}$ denotes the subspace of valuations that can be represented as in \eqref{eq:smooth-val} by smooth differential forms. The following fundamental result was proved by Alesker \cite{Alesker:manifoldsI}. We refer the reader to \cite{HK:localization} for a recent more elementary proof of this result. 

\begin{theorem} $\Val_k^{sm}= \Val_k^\infty$. 
\end{theorem} 

Our ultimate goal is to prove the corresponding statement for curvature measures. 

\begin{lemma}
$\Curv^{sm}_k\subset \Curv^\infty_k$. 
\end{lemma}
\begin{proof} There is nothing to prove for $k=n$. 
For $k\in\{0,\ldots, n-1\}$ define 	
		$$ I\colon \Omega^{k,n-1-k}( S^* \RR^n)^{tr} \otimes or(\RR^n ) \to \Curv_k(\RR^n)$$
		by 
		$$ I(\omega)(K,\beta)= \int_{\nc(K)} \mathbf{1}_{\pi^{-1}(\beta)} 
	\, \omega.$$
	Here $or(\RR^n)$ denotes the space of functions $\sigma\colon \Lambda^n\, \RR^n\setminus \{0\} \to \CC$ satisfying $\sigma(t w)= \operatorname{sgn}(t) \sigma(w)$ for $t\neq 0$. 
	Observe that the restriction of $\omega$ to the normal cycle of $K$ defines a density which can be integrated without a choice of orientation of $\nc(K)$. This shows that $I$ is $GL(n,\RR)$ equivariant.

	For any compact subset $\calF$ of $C_0(\RR^n)$, 
	$$ p_{B,\calF}(I(\omega)) \leq   \sup_{f\in \calF} \|f\| \cdot \sup_{K\subset B} \mathbf{M}(\nc(K))  \cdot   \sup_{\pi^{-1}(B)} |\omega|.$$
	Since by Lemma~\ref{lemma:prop-nc} the supremum $\sup_{K\subset B} \mathbf{M}(\nc(K))$ is finite, we conclude that $I$ is continuous. 
	
	By Lemma~\ref{lemma:smooth-sec}, the representation of $GL(n,\RR)$ on $\Omega^{k,n-1-k}( S^* \RR^n)^{tr} \otimes or(\RR^n )$ is smooth. Since $I$ is a $GL(n,\RR)$  homomorphism, it follows that $I(\omega)$ is a smooth vector of $\Curv_k(\RR^n)$. 
\end{proof}

The reverse inclusion 
\begin{equation} \label{eq:sm-infty}
	\Curv^\infty_k \stackrel{?}{\subset} \Curv^{sm}_k,
\end{equation} 
looks innocent, but it is not. Recall that the smooth vectors of any representation  form a dense subspace. Therefore, if \eqref{eq:sm-infty} were true, we would immediately obtain that smooth curvature measures  are dense, and there is no evident reason why that should hold. The following proposition provides a sufficient condition.
\begin{proposition} \label{prop:Z-irred}
	If $\Z_k^\pm$ is irreducible, then $\Curv_k^{sm}= \Curv^{\infty}_k$.  
\end{proposition}

For the proof we need the following fact from \cite{BernigBroecker:Rumin}.

\begin{lemma} \label{lemma:ker-int}
Let $\omega\in \Omega^{n-1}(S\RR^n)$. If, for all convex bodies $K$ and Borel sets $\beta\subset \RR^n$,
$$ \int_{\nc(K)} \mathbf{1}_{\pi^{-1}(\beta)} \, \omega=0,$$
then $\omega$ belongs to the differential ideal $\mathcal I \subset \Omega^*(S\RR^n)$ generated by the contact form 
$$ \alpha_{(x,u)}(X)  =  \langle d\pi(X), u\rangle, \quad X\in T_{(x,u)} S\RR^n.$$
\end{lemma}

\begin{lemma}\label{lemma:notin-I}Let $\mathcal I \subset \Omega^*(S\RR^n)$ denote the differential ideal generated by the contact form. At the point $u=e_n$, the form
\begin{equation}\label{eq:primitive-form}  dx_1\wedge \cdots \wedge dx_{k} \wedge du_{k+1}\wedge \cdots \wedge du_{n-1}\end{equation}
does not belong to $\calI_{(x,u)}= \{ \omega_{(x,u)}\colon \omega\in \mathcal I\}$.
\end{lemma}
\begin{proof}
	The tangent space of $S\RR^n$ at $(x,u)$ splits as $\RR (u,0) \oplus \ker \alpha$. For $u=e_n$, we  can identify 
	$ \ker \alpha $ with $ \RR^{n-1} \oplus \RR^{n-1}$. Under this identification  $d\alpha = \sum_{i=1}^{n-1} du_i \wedge dx_i$, while $\alpha = dx_n$. The statement we want to prove can be reduced to symplectic linear algebra, as follows. Since  $Q=  \RR^{n-1} \oplus \RR^{n-1} $ together with $d\alpha$ is a symplectic vector space, the Lefschetz decomposition yields 
	$$\Lambda^{n-1} Q^* = P \oplus  d\alpha  \wedge \Lambda^{n-3} Q^*, $$
where $$ P= \{ \eta\in \Lambda^{n-1} Q^* \colon d\alpha \wedge \eta=0\}$$ 
denotes the subspace of primitive forms. 
Since \eqref{eq:primitive-form} is clearly primitive, we conclude that it cannot be a multiple of $d\alpha$. 
\end{proof}

\begin{proof}[Proof of Proposition~\ref{prop:Z-irred}] Since the topology on $\Curv(\RR^n)$ is not metrizable, we cannot directly apply the Casselman--Wallach theorem (Theorem~\ref{thm:CW}). Instead we argue as follows.  Let $k\in \{0,\ldots, n-1\}$ and  denote by $\Omega^+_{i,j}$ and $\Omega^-_{i,j}$ the  $+1$ and $-1$  eigenspaces of the action of $-\id$ on $\Omega^{i,j}( S^* \RR^n)^{tr} \otimes or(\RR^n )$. 
	Since $\Val_{k+1}(\RR^n,\CC^n)^\infty$ satisfies the hypothesis of the Casselman--Wallach theorem  by Lemma~\ref{lemma:CW-thm}, the image of the $GL(n,\RR)$ equivariant map
	$$   J=B \circ I\circ d\colon \Omega_{k,n-2-k}^\pm \to \Val_{k+1}^\pm (\RR^n,\CC^n)^\infty$$
	is closed. 
	
	We claim that the image of $J$ is distinct from $\{0\}$.  Indeed, by Lemma~\ref{lemma:notin-I}, a function $f$ can  be chosen so that $d\omega$, where    
	$$ \omega =  f(u)\, dx_1\wedge \cdots \wedge dx_{k} \wedge du_{k+1}\wedge \cdots \wedge du_{n-2} \in \Omega_{k,n-2-k}^\pm , $$
	does not belong to the differential ideal generated by the contact form. Hence,  by Lemma~\ref{lemma:ker-int}, $I(d \omega)$ is nonzero and belongs to  $\Z_k^\pm$. By the injectivity of the Bernig embedding, $J(\omega)$ is nonzero.

	Let $W$ denote the closure of $B(\Z_k^{\pm,\infty})$ in  $\Val_{k+1}^\pm (\RR^n,\CC^n)^\infty$. Observe that $W_K \cong (\Z_k^{\pm})_K$ as $(\fgg,K)$ modules. Therefore, if $\Z_k^{\pm}$ is irreducible, then so is $W$. Since $\im(J)\subset W$, we conclude that $\im(J)=W$.

	Now let $\Phi\in \Curv^{\pm,\infty}_k$. Then $\glob(\Phi)$ is a smooth valuation and hence there exists a differential form $\omega$ such that $\glob(\Phi- I(\omega))=0$. By the previous paragraph, 
	$ B(\Phi- I(\omega))\in W=\im(J)$ and thus there exists a  form $\eta$ such that 
	$$ B(\Phi- I(\omega + d\eta))=0.$$
	From the injectivity of the Bernig embedding we conclude that $\Phi=  I(\omega + d\eta)$. 
\end{proof}

While the Casselman--Wallach  theorem appears not to be  applicable in the following situation,  we still obtain the same conclusion.

\begin{corollary}
If $\Z_k^\pm$ is irreducible, then the image of $B\colon \Z_k^\infty\to \Val_{k+1}(\RR^n,\CC^n)^\infty$ is closed. 
\end{corollary}
\begin{proof}
	The proof of Proposition~\ref{prop:Z-irred} shows that $$\overline {B(\Z_k^{\pm,\infty})} = W = \im J \subset B(\Z_k^{\pm,\infty}).$$ 
	Since the projections in  $\Val_{k+1}(\RR^n,\CC^n)^\infty$ onto the subspaces of even and odd valuations are continuous, we conclude that $B(\Z_k^\infty)$ is closed. 
\end{proof}

\subsection{Curvature measures on the plane}

In this section, we consider curvature measures in $\RR^2$ and give a proof for $\Curv^{sm}_k = \Curv^\infty_k$ that does not rely on the irreducibility of $\Z_k^\pm$.   By the results of  Section~\ref{sec:degree}, it remains to treat the case of  curvature measures of degree $k=0$.

Let us write  $ \partial_\phi h :=  h_\phi: =  \frac{d}{d\phi} h(e^{ i \phi})$ for smooth functions  on the unit circle. Let $\rho_{\pi/2} u = (-u_2,u_1)$ denote a rotation of $u\in \RR^2$ by the angle $\pi/2$.
\begin{lemma}\label{lemma:byparts} Let $h\colon S^1\to \CC$ be a smooth  function satisfying $\int_{S^{1}} h(u)\, du=0$. Then the curvature measure 
		$$ \Phi(K,\beta) = \int_{\pi^{-1}(\beta)} h(u)\, d\Theta_0(K,(x,u))$$ 
globalizes to zero and 
$$ B(\Phi)(K)= -\int_{S^{1}} h_K(u) (h_\phi (u) \cdot \rho_{\pi/2} u   -2 h(u)\cdot u)\, du.$$ 
\end{lemma}

The above lemma is a special case of the following result that holds in all dimensions.
Let us write $\grad f$ for the gradient of a smooth function $f\colon U\to \CC$ defined on an open subset of $\RR^n$. 

\begin{proposition}
	Let $f\colon \RR^n\setminus\{0\}\to \CC$  be a smooth $(-n)$-homogeneous function satisfying $\int_{S^{n-1}} f(u)\, du =0$. Then 
	the curvature measure 
		$$ \Phi(K,\beta) = \int_{\pi^{-1}(\beta)} f(u)\, d\Theta_0(K,(x,u))$$ 
	globalizes to zero and 
	$$ B(\Phi)(K)= -\int_{S^{n-1}}     h_K(u) \grad f(u)\, du.$$ 
\end{proposition}
\begin{proof}
It suffices to prove the expression for $B(\Phi)(K)$ for convex bodies $K$ with smooth and strictly positively curved boundary. For such bodies the Gauss map $n_K\colon \partial K\to S^{n-1}$ is a diffeomorphism and $n_K^{-1}(u)= \grad h_K(u)$. By \cite[Lemma 4.2.2]{Schneider:BM},
$$ \Phi(K,g)  = \int_{ S^{n-1}}  g(n_K^{-1}(u)) f(u)\,  du $$ 
and hence 
$$ B(\Phi)(K)= \int_{ S^{n-1}}  f(u) \grad h_K(u) \, du.$$

If $M$ is a compact Riemannian manifold without boundary, 
$$ 0 = \int_M \div(\varphi X)= \int_M X\varphi  + \varphi \div X$$
for every smooth function $\varphi$ and vector field $X$ on $M$. 
We will apply this identity in the case $M=S^{n-1}$ to the vector field
$ X_u= x - \langle u,x\rangle u$, where $x$ is a fixed vector in $\RR^n$.
Since 
$$ \div X=-  (n-1) \langle u,x\rangle,$$
we conclude 
$$ 0=\int_{S^{n-1}} \langle \grad_{S^{n-1}} \big(  \varphi, x\rangle -  (n-1) \langle u,x\rangle \varphi (u)\big)\,  du= \int_{S^{n-1}} \langle \grad\varphi ,x\rangle \, du$$
if $\varphi$ is extended to a $-(n-1)$-homogeneous function on $\RR^n\setminus\{0\}$. 

Since $\varphi= h_K f$ is homogeneous of degree $-(n-1)$, we obtain
$$  0= \int_{S^{n-1}} f(u) \grad h_K(u)  + h_K(u) \grad f(u)\, du$$
and hence the desired conclusion.
\end{proof}

\begin{theorem}
	For each  $\Phi\in \Curv^\infty_0(\RR^2)$, there exists a unique smooth function $h\colon S^1\to\CC$  such that 
	$$ \Phi(K,A) = \int_{\pi^{-1}(A)} h(u)\, d\Theta_0(K,(x,u)).$$ 
	
\end{theorem} 
\begin{proof} If $\Phi$ is $GL(2,\RR)$ smooth, then so is $B(\Phi)$. Hence the function $f$ appearing in  \eqref{eq:B-f} is smooth, see Theorem~A.2 in the appendix of \cite{Schuster:log}.
 By Proposition~\ref{prop:imB}, we may modify $f$ by a linear map $\RR^2\to \CC^2$ so that 
$$  \langle f(u),u\rangle =0$$
holds  for all $u\in S^1$. Consequently, we have $f(u)= g(u) \cdot \rho_{\pi/2} u$ for a smooth function $g$ on the unit circle. 

 Put $h= -g_\phi$ and define 
 $$ \Phi_h(K,A) =\int_{\pi^{-1}(A)} h(u)\, d\Theta_0(K,(x,u)).$$ 
 We claim that $\Phi= \Phi_h$. Indeed, using Lemma~\ref{lemma:byparts} we obtain

\begin{align*} B(\Phi_h)(K) & =  - \int_{S^1} h_K ( h_\phi  \cdot  \rho_{\pi/2} u  -2 h \cdot u)\,   du\\
  & =   \int_{S^1} h_K ( g_{\phi\phi}  \cdot   \rho_{\pi/2} u  -2 g_\phi \cdot u)\,   du\\
 & =   \int_{S^1} h_K ( \Delta f + f  )\,    du\\
 & =   \int_{S^1} f(u)\,  d\!\S_1(K,u)\\
 & =  B(\Phi)(K).
\end{align*} 
Here $\Delta f= f_{\phi\phi}$ and we used the identity $d\! \S_1(K,u) = (\Delta h_K+ h_K)\, du$ for $C^2_+$ convex bodies, see \cite[eq. (2.56)]{Schneider:BM}. 
Since the Bernig embedding is injective, we obtain $\Phi_h = \Phi$, as desired. 
\end{proof}

\section{Homogeneous functions on $\RR^n$}

In this section we treat a seemingly unrelated problem: we will investigate the irreducibility and composition series of certain spaces of homogeneous functions under the action of the general linear group. 

\begin{definition}
	Let $\calF_\alpha$ denote the vector space of $C^\infty$-smooth functions  
	$f\colon \RR^n\setminus \{0\} \to \CC$ that are homogeneous of degree $\alpha\in \RR$. 
	We let the group $GL(n,\RR)$ act on $\RR^n$ by 
	$$ \sigma(g)x= g^{-\top} x$$ and 
	on functions by 
	$$ \pi(g) f(x)= f(g^\top x).$$

Similarly, we let $\calV_\alpha$ denote the vector space of $C^\infty$-smooth functions  
$F\colon \RR^n\setminus \{0\} \to V^*$ that are homogeneous of degree $\alpha\in \RR$. Here $V= \CC^n$ denotes the complexification of $\RR^n$. 
We let the group $GL(n,\RR)$ act on these functions by 
$$ \pi(g) F(x)= \sigma(g)^{-\top} F(g^\top x).$$	
Let $\calV_\alpha^+$ and $\calV_\alpha^-$ denote the subspaces of even $\pi(-\id )F =F$ and odd $\pi(-\id )F =-F$ functions. 
\end{definition}

The following lemma is well known.
\begin{lemma} \label{lemma:ell}
	The linear functional defined on $\calF_{-n}$ by 
	$$ \ell(f)= \int_{S^{n-1}} f(u)\,  du$$
	has the property
	$$ \ell (\pi(g) f) =  |\det g|^{-1}\,  \ell(f), \quad g\in GL(n,\RR).$$
\end{lemma}

By the previous lemma, the closed subspace $\calU \subset \calV_{-(n+1)}$   defined by 
$$ \ell ( \xi F)=0$$
for every linear functional $\xi\colon \RR^n\to \CC$, is $GL(n,\RR)$ invariant.
Let $\calU^+$ and $\calU^-$ denote the subspaces of even $\pi(-\id )f =f$ and odd $\pi(-\id )f =-f$ functions. 
The following is the  main result of this section.

\begin{theorem}\label{thm:length2}
	
For $n\geq 2$, the representation of $GL(n,\RR)$ on 	$\calU^\pm$ has   length $2$. 
	
\end{theorem} 

Let us  show how Theorem~\ref{thm:length2} implies Theorem~\ref{thm:Val2}.

\begin{proof}[Proof of Theorem~\ref{thm:Val2}] For every  $F\in \calU$ and every convex body $K\subset \RR^n$ define
	
	 $$ \varphi_F(K)= \int_{S^{n-1}} h_K(u) F(u) \, du.$$
$\varphi_F$ evidently satisfies 
$$ \varphi_F(K+L)= \varphi_F(K)+ \varphi_F(L)$$
for all convex bodies  $K$ and $L$. As is well-known, see, e.g., \cite[p. 186]{Schneider:BM},  this implies that $\varphi_F$ is a valuation. For $x\in \RR^n$, 
$$ \varphi_F(K+x)= \int_{ S^{n-1}} (h_K(u)+ \langle x,u\rangle) F(u) \, du = \varphi_F(K)$$ 
by the definition of $\calU$. It follows that $\varphi_F\in \Val_1(\RR^n, V^*)$. 

We now slightly modify the definition of $\varphi_F$ to ensure correct behavior under the action of $GL(n,\RR)$. Let $\Dens((\RR^n)^*)$ denote the $1$-dimensional space of densities on $(\RR^n)^*$ with a  representation of $GL(n,\RR)$ defined  by
$$ g\cdot \mu= |\det g| \mu.$$
Given such a density $\mu$, we define a density on $S^{n-1}\subset (\RR^n)^*$ by 
$$\mu_{S^{n-1}}(w)= \mu(u\wedge w), \quad w\in \Lambda^{n-1}\, T_uS^{n-1}.$$
One immediately verifies that  
$$ T\colon \calU\otimes \Dens((\RR^n)^*) \to \Val_1(\RR^n,V^*), \quad F\otimes \mu \mapsto \int_{S^{n-1}}  h_K F  \mu_{S^{n-1}},$$
is well-defined, continuous, and $GL(n,\RR)$ equivariant. Theorem~\ref{thm:1-hom} and the open mapping theorem imply that the induced map 
$$ T\colon \calU^\pm\otimes \Dens((\RR^n)^*) \to \Val_1^\pm(\RR^n, V^*)^\infty $$ 
is a $GL(n,\RR)$ isomorphism. Since tensoring with a $1$-dimensional representation does not change the length of a representation, we conclude that 
$\Val_1^\pm(\RR^n,\CC^n)$ has length $2$. 

The  corresponding statement for $\Val_{n-1}^\pm(\RR^n,\CC^n)$ now follows immediately from the fact that the Fourier transform of valuations (see \cite{Alesker:Fourier} or \cite{FW:Fourier}) gives rise to a $GL(n,\RR)$ isomorphism
$$ \Val_1^\pm (\RR^n,\CC^n)^\infty \to \Val_{n-1}^\pm ((\RR^n)^*,\CC^n)^\infty \otimes \Dens(\RR^n).$$
	 \end{proof}

\subsection{Roots and weights}

Let $V$ and $Q$ denote the complexification of $\RR^n$ and its standard Euclidean inner product.
Let $e_1,\ldots, e_n$ denote the standard basis of $V=\CC^n$. We define another basis of $V$
by 
$$ v_j =  e_{2j-1} - i e_{2j},\quad v_{n+1-j}  = \b v_j= e_{2j-1}+ i e_{2j},$$
for $1\leq j \leq m= \lfloor n/2 \rfloor$. If $n$ is odd, we also set 
$$ v_{m+1} =  \sqrt 2    e_n.$$
With respect to this basis, the defining equation of the Lie algebra $\so(n)_\CC$,  the complexification of $\so(n)$, becomes 
$$ X^\top J + JX=0$$
where $J$ is the matrix with antidiagonal $(2,\ldots, 2)$ and zeros elsewhere.  
A basis of $\so(n)_\CC$ is therefore given by 
\begin{align}\label{eq:roots} \begin{split} e_{ab} & = E_{ab} - E_{n+1-b,n+1-a} \quad (1\leq a,b\leq m),\\
	s^+_{ab} & =  E_{b,n+1-a} -  E_{a,n+1- b}\quad (1\leq a<b\leq m),\\
		s^-_{ab} & = E_{n+1-a,b} - E_{n+1-b,a} \quad (1\leq a<b\leq m),
		\end{split}
\end{align}
and, if $n$ is odd, 
\begin{align} \label{eq:roots-odd}  \begin{split} f^+_ a  & = E_{a, m+1} -  E_{m+1,n+1-a}\quad (1\leq a\leq m),\\
 f^-_ a  & =  E_{ m+1, a} -  E_{n+1-a,m+1}\quad (1\leq a\leq m).
 \end{split}
\end{align}
Here  $E_{j,k}$ is the matrix with a $1$  in the 
$(j,k)$ entry and zeros elsewhere.

A Cartan subalgebra $\mathfrak h$ is spanned by $e_{jj}$ ($1\leq j\leq m$). We define $\epsilon_j \colon \mathfrak h \to \CC$ by $\epsilon_j(e_{kk})= \delta_{jk}$.
Observe  that 
$e_{ab}$ spans the root space of $
\epsilon_a-\epsilon_b$,  $s^+_{ab}$ and $s^-_{ab}$ span the root spaces of the ``sum roots'' $\epsilon_a+ \epsilon_b$ and $-(\epsilon_a+ \epsilon_b)$, and $f^+_a$ and $f_a^-$ span the root spaces of the ``short roots'' $\epsilon_a$ and $-\epsilon_a$. 

Note   that for $1\leq j\leq m$ the vector  $v_j$ is a weight vector of weight $\epsilon_j$ and $\b v_j$ is a weight vector of weight $-\epsilon_j$. This fact  can be used to easily verify the  root space decomposition.

A Borel subalgebra $\mathfrak b$ is spanned by $e_{ab}$ ($1\leq a\leq b\leq m$) and $s^+_{ab}$ ($1\leq a<b \leq m$) (and, if $n$ is odd, by $f_a^+$, $1\leq a\leq m$).

Let $\gl(n,\RR)= \so(n)\oplus \mathfrak p$ denote the Cartan decomposition into skew-symmetric and symmetric matrices.
If we  define as usual  $uv= \frac 12 (  u\otimes  v + v\otimes u)$, then $T\colon \Sym^2 \RR^n \to \mathfrak p$, defined by  
$$\langle T(uv ) x,y\rangle = \frac 12 ( \langle u,x\rangle \langle v,y\rangle +\langle u,y\rangle \langle v,x\rangle),$$ is an $O(n)$ equivariant isomorphism.
Its complexification $T\colon \Sym^2 V\to \mathfrak p_\CC$ satisfies    
$$ T(v_jv_k)=E_{j,n+1-k}+E_{k,n+1-j}$$
and hence defines via
\begin{equation}\label{eq:T-p} T(v_a v_b) = p_{ab}^+, \quad T(\b v_a \b v_b)= p_{ab}^-, \quad T(v_a \b v_b) = 
	q_{ab},\end{equation}
where $1\leq a,b\leq \lfloor (n+1)/2\rfloor$,  a convenient basis 
of $\mathfrak p_\CC$.

Recall that the representations of the group $O(n)$ are parametrized by certain Young diagrams. Let $\Lambda^+(n)$ denote the set of $m$-tuples  $\lambda= (\lambda_1,\ldots, \lambda_m)$, $m= \lfloor n/2\rfloor$,  of nonnegative integers $\lambda_1\geq  \cdots \geq \lambda_m$. We identify $\lambda$ with the Young diagram $(\lambda_1,\ldots, \lambda_m, 0,\ldots)$ and denote the corresponding representation of $O(n)$ by $\rho^\lambda$.

The highest weights of representations of $\so(n)_\CC$ with respect to $\mathfrak b$ corresponding to representations of $SO(n)$ are the integer linear combinations 
$\mu=  \mu_1\epsilon_1 + \cdots + \mu_m \epsilon_m$ such that 
$$ \begin{cases}
\mu_1\geq \cdots \geq \mu_{m}\geq 0, & \text{if } n \text{ is odd}, \\
\mu_1\geq \cdots \geq \mu_{m-1} \geq  |\mu_{m}|, & \text{if } n \text{ is even.}
\end{cases}
$$
We shall denote the representations of $SO(n)$ corresponding to $\mu$ by $\pi^\mu$. 

If $n$ is even,  we define $\mu^-=(\mu_1,\ldots, \mu_{m-1}, -\mu_m)$ for $\mu\in \Lambda^+(n)$. In the following, we will often identify the $m$-tuples $\mu$ and $\mu^-$ with the corresponding highest weights. 

The relationship between the irreducible representations of $O(n)$ and $SO(n)$ is as follows. If 
\begin{enuma} 
	\item $n$ is odd; or if
	\item $n$ is even and $\mu_m=0$,
\end{enuma}
then $\rho^\mu$ is irreducible  as an $SO(n)$ representation and 
$ \rho^\mu = \pi^\mu$. 
However, if $n$ is even and $\mu_m\neq 0$, then $$\rho^\mu|_{SO(n)} =  \pi^\mu \oplus \pi^{\mu^-}.$$

\subsection{Construction of highest weight vectors}

We assume $n\geq 3$ throughout.
	Recall that we  let the group $GL(n,\RR)$ act on $\RR^n$ and its complexification $V$ by 
$ \sigma(g)x= g^{-\top} x$.
Observe that $E=\coprod_{L}V/L $, where $L$ ranges over all complex lines through the origin in $V$, is the total space of a smooth vector bundle over projective space. Hence we may speak of smooth maps with values in $E$. 

\begin{definition}

Let $\calX_\alpha$ denote the vector space of $C^\infty$-smooth functions  
$F\colon \RR^n\setminus \{0\} \to V^*$ that are homogeneous of degree $\alpha$ and satisfy 
$$ \langle  F(x),x\rangle = 0.$$
We let the group $GL(n,\RR)$ act on these functions by 
$$ \pi(g) F(x)= \sigma(g)^{-\top} F(g^\top x).$$	
Let $\wt \calX_\alpha$ denote the vector space of $\alpha$-homogeneous  $C^\infty$-smooth functions  $G\colon \RR^n\setminus \{0\}\to \coprod_{L}V/L $, where $L$ ranges over all complex lines through the origin in $V$, satisfying 
$$ G(x)\in V/\CC x.$$ 
The group $GL(n,\RR)$ acts on these functions by 
$$ \pi(g) G(x)= \sigma(g) G(g^\top x).$$	
We denote by $\calX^\pm_\alpha$ and $\wt \calX^\pm_\alpha$ the $+1$ and $-1$ eigenspaces under the action of $-\id$. 
\end{definition}

We define a basis of $V^*$ by 
$$ z_j=x_{2j-1} - i x_{2j}, \quad z_{n+1-j}= \b z_j=x_{2j-1} + i x_{2j}$$ 
for $1\leq j \leq \lfloor n/2 \rfloor$ and, if $n$ is odd, by 
$$ z_{m+1} = \b z_{m+1} = \sqrt2 x_n.$$
By abuse of notation,  $x_1,\ldots, x_n$ denote both the canonical coordinates on $\RR^n$ and  their complex linear extension to $V$. 
Observe that $\sum_{j=1}^n z_j \b z_j= 2 \sum_{j=1}^n x_j^2$  as functions on $\RR^n$.

We also define the differential operators
$$ \frac{\partial }{\partial z_j} = \frac{1}{2} \Big( \frac {\partial }{\partial x_{2j-1}} +i \frac {\partial }{\partial x_{2j}}\Big), \quad 
\frac{\partial }{\partial z_{n+1-j}} = \frac{\partial }{\partial \b z_{j}} =  \frac{1}{2} \Big( \frac {\partial }{\partial x_{2j-1}} -i \frac {\partial }{\partial x_{2j}}\Big)$$
for $1\leq j \leq \lfloor n/2 \rfloor$  and, if $n$ is odd,  
$$ \frac{\partial }{\partial z_{m+1}} = \frac{1}{\sqrt 2} \frac{\partial }{\partial x_n}.$$
The  rationale behind these definitions is of course that the complexification of the differential of a function $f$ satisfies 
$$ df_p = \sum_{j=1}^n  \frac{\partial f }{\partial z_j}(p)  z_j \in V^*.$$

\begin{lemma} 
The Lie algebra $\so(n)_\CC$ acts on functions  from $\calF_\alpha$ by the operators
\begin{align*}e_{ab} &  = z_a \frac{\partial }{\partial z_b} - \b z_b \frac{\partial }{\partial \b z_a},\\ 
s_{ab}^+ &  = z_b \frac{\partial }{\partial \b z_a} -  z_a \frac{\partial }{\partial \b z_b},\\ 
s_{ab}^- &  = \b z_a \frac{\partial }{\partial z_b} - \b z_b \frac{\partial }{\partial z_a}.
\intertext{If $n$ is odd, we also have the operators}
f_{a}^+ &  = z_a \frac{\partial }{\partial  z_{m+1}} -  z_{m+1} \frac{\partial }{\partial \b z_a},\\ 
f_{a}^- &  =  z_{m+1} \frac{\partial }{\partial z_a} - \b z_a \frac{\partial }{\partial z_{m+1}}.
\end{align*} 
The  noncompact part  $\mathfrak p_\CC$ acts by the operators
\begin{align*}
	p_{ab}^+ &  = z_a \frac{\partial }{\partial \b z_b} +  z_b \frac{\partial }{\partial \b z_a},\\ 
	p_{ab}^- &  = \b z_a \frac{\partial }{\partial z_b} + \b z_b \frac{\partial }{\partial z_a},\\
	q_{ab} &  = z_a \frac{\partial }{\partial z_b} + \b z_b \frac{\partial }{\partial \b z_a}.
\end{align*} 
\end{lemma}  
\begin{proof}  Since $\pi(X) f (p) = df_p ( X^\top p)$, $X\in \gl(n,\RR)$,  the computation is reduced to changing the basis from $v_1,\ldots,v_n$ to $\partial/\partial z_1, \ldots, \partial/\partial z_n$, which is dual to $z_1,\ldots, z_n$. The effect of this change of basis  is that 	$E_{j,k}  \longleftrightarrow E_{\b j, \b k}$, where $\b j =n+1-j$ for brevity.
	Using this, it is straightforward to read off the differential operators  directly from the definition of the matrices \eqref{eq:roots}, \eqref{eq:roots-odd},  and \eqref{eq:T-p}. 
\end{proof}	
Define $r^2(x)=2|x|^2= \sum_{j=1}^n z_j\b z_j$.
\begin{lemma}\label{lemma:hwv}
\begin{enuma}
	\item For $k\geq 1$, the functions 
	\begin{align*}
		\xi_{k}(x) &=   r^{\alpha -k+1} z_1^{k-1}(x)  z_1 -  r^{\alpha-k-1}  z_1^{k}(x) \sum_{j=1}^n  z_j(x) \b z_{j}\\ 
		\eta_{k}(x) &=  r^{\alpha-k}(x) z_1^{k-1}(x) ( z_2(x) z_1 -z_1(x) z_2) 
	\end{align*}
	in $\calX_{\alpha}$ are  highest weight vectors. If $n=3$, both have weight $k\epsilon_1$. If $n\geq 4$, they have  weight $k \epsilon_1$ and $k \epsilon_1+ \epsilon_2$. 
		\item Let $P_{V/\CC x} \colon V\to V/\CC x$ denote the canonical projection. The functions 
	\begin{align*}
		\wt\xi_k(x) &=   r^{\alpha -k+1} z_1^{k-1}(x) P_{V/\CC x} (v_1)\\ 
		\wt \eta_k(x) &=  r^{\alpha-k}(x) z_1^{k-1}(x) P_{V/\CC x} ( z_2(x) v_1 -z_1(x) v_2) 
	\end{align*}
	in $\wt \calX_{\alpha}$ are  highest weight vectors.
	If $n=3$, both have weight $k\epsilon_1$. If $n\geq 4$, they have  weight $k \epsilon_1$ and $k \epsilon_1+ \epsilon_2$. 	
\end{enuma} 
\end{lemma} 
\begin{proof}
	We have to show that the vectors in question are annihilated by $e_{ab}$ and $s_{ab}^+$ for $1\leq a<b\leq m$ and, if $n$ is odd, by  $f_a^+$ for $1\leq a\leq m$. 
	For (a)  this is straightforward.

	For (b) observe that $\sigma(g) \circ  P_{V/\CC g^\top x}= P_{V/\CC x} \circ  \sigma(g) $ and hence   $$\pi(g)(\wt\xi_k)(x) = (r^{\alpha-k+1} z_1^{k-1})(g^\top x) P_{V/\CC x}(\sigma(g) v_1),$$
	which can be immediately differentiated with respect to $g$. Similar reasoning applies to  $\wt \eta_k$. 
\end{proof}

\begin{remark}
For $n=4$  there exist also highest weight vectors  $\eta_k^-$ of weight $k \epsilon_1- \epsilon_2$. 
	However, since  $\eta_k$ and $\eta_k^-$ belong to the same irreducible subrepresentation of $K=O(n)$, we only need one of them to determine the $GL(n,\RR)$ module structure of $\calX_\alpha^\pm$.
\end{remark}

\begin{corollary} \label{cor:K-types} For $n\geq 3$, the representations  $\calX_\alpha^\pm|_K$ are multiplicity free. The spaces of $K$-finite vectors 
	are 
	\begin{align*} \calX_{\alpha}^{+,K} & =   \bigoplus_{l\geq 1} V_{2l} \oplus  W_{2l-1}\\
		\calX_{\alpha}^{-,K} & =   \bigoplus_{l\geq 1} V_{2l-1} \oplus  W_{2l},
		\end{align*} 
	where $V_k \cong \rho^{(k,0)}$ and $W_k  \cong \rho^{(k,1)}$ for $n\geq 4$ and $W_k \cong V_k\cong \rho^{(k)}$ for $n=3$. 
\end{corollary}
\begin{proof}Consider the space
	$$ \mathrm{Ind}_{SO(n-1)}^{SO(n)} \CC^{n-1} = \left\{ f\in C^\infty( SO(n), \CC^{n-1})\colon  f(kl) = l^{-1} f(k) \text{ for all } l\in SO(n-1)\right\}.$$
	The group $SO(n)$ acts on these functions by left translation 
	$ (k\cdot f)(l)= f(k^{-1}l)$.
	
	For $F\in \calX_\alpha$ consider the function $\wt F\colon SO(n)\to  e_1^\perp \subset V^*$, defined by 
	$$ \wt F (k)=  k^{-1} F(k e_1).$$
	Then $F\mapsto \wt F$ defines an $SO(n)$ equivariant isomorphism $\calX_\alpha \cong \mathrm{Ind}_{SO(n-1)}^{SO(n)} \CC^{n-1}$. 
	
	Suppose $n\geq 4$. 
	By Frobenius reciprocity, the multiplicity of the $SO(n)$ representation $\pi^\mu$ in $\calX_\alpha|_{SO(n)}$ equals the multiplicity of $\CC^{n-1}$ in $\pi^\mu|_{SO(n-1)}$. 
	By the branching rule from $SO(n)$ to $SO(n-1)$, see \cite[Theorem 9.16]{Knapp:beyond}, $\pi^\mu$ decomposes with multiplicity $1$ and $\CC^{n-1}$ appears if and only if 
	$$ \mu_1 \geq 1 \geq \mu_2 \quad \text{and}   \quad    \mu_3=\cdots = \mu_m=0.$$
	
	If $n=3$, then $\CC^{n-1}$ is no longer irreducible, but decomposes as $\pi^{(1)}\oplus \pi^{(-1)}$. Frobenius reciprocity and the branching rule applied to both summands imply that, for  $k\geq 1$, the multiplicity of $\rho^{(k)}$  is $2$. 
	
	The  highest weight vectors of Lemma~\ref{lemma:hwv} belong to either $\calX_\alpha^+$ or $\calX_\alpha^-$. This establishes the claimed decompositions.	
\end{proof}

\subsection{Transition between $K$-types}

In this subsection, we investigate the $(\fgg, K)$ modules   $(\calX^{\pm}_{\alpha})_K$. We will follow the strategy of \cite{HoweLee:degenerate}, which can be described as follows. For each $K$ type $V_\mu$, where, according to Corollary~\ref{cor:K-types}, $\mu=(k,0,\ldots,0)$ or $\mu=(k,1,0,\ldots, 0)$, consider the $K$ representation $V_\mu\otimes  \mathfrak p_\CC$ and the equivariant map 
$$ m_\mu\colon   V_\mu\otimes  \mathfrak p_\CC \to (\calX_\alpha^\pm)_K, \quad m_\mu(v\otimes X) = \pi(X)v.$$
If $w$ is a highest weight vector of weight $\lambda$ in $V_\mu\otimes \mathfrak p_\CC$, then $m_\mu(w)$ is a highest weight vector in $V_\lambda$ or zero. Note that if $m_\mu(w)$ is nonzero, then $V_\mu$ can be transformed to $V_\lambda$ under the action of $\fgg$.  

It is straightforward to investigate the  transition from $V_{\mu}$ to $V_{\mu +2\epsilon_1}$. By passing to the Hermitian dual as in \cite{HoweLee:degenerate}, we will obtain information about the  transition from $V_{\lambda}$ to $V_{\lambda-2\epsilon_1}$. To establish the irreducibility of (certain subquotients of) $\calX_{\alpha}^\pm$ under the action of $GL(n,\RR)$, it will then suffice to explicitly construct highest weight vectors in $V_\mu\otimes \mathfrak p_\CC$ for a small number of $K$-types $V_\mu$. 
Since the $GL(n,\RR)$ module structure  of  $\calF_\alpha^\pm$ is well known,  we will be able to complete  the proof of Theorem~\ref{thm:length2}. 

\subsubsection{Transition from $V_\mu$ to $V_{\mu+2\epsilon_1}$}

Observe that for every $K$-type $V_\mu$ with highest weight vector $u_\mu$, $$u_\mu \otimes v_1^2\in V_\mu\otimes \Sym^2 V$$
is a highest weight vector of weight $\mu+2\epsilon_1$. 
 Therefore  $$p_{11}^+ u_\mu= m_\mu(u_\mu \otimes T(v_1^2))$$ is either zero or a highest weight vector. 

\begin{lemma}\label{lemma:+2}
	For $k\geq 1$ 
	\begin{enuma}

\item 
\begin{align*}
	p_{11}^+ \xi_{k} &= 2(\alpha-k-1) \xi_{k+2}\\
	p_{11}^+ \eta_{k} &= 2(\alpha-k) \eta_{k+2}
\end{align*}   
\item 
\begin{align*}
	p_{11}^+ \wt\xi_{k} &= 2(\alpha-k+1) \wt\xi_{k+2}\\
	p_{11}^+ \wt \eta_{k} &= 2(\alpha-k) \wt \eta_{k+2}
\end{align*}   

	\end{enuma}
\end{lemma}
\begin{proof}
	This follows from a straightforward computation using 
	\begin{equation}\label{eq:p+r}p_{ab}^+ (r^\alpha)=  2 \alpha r^{\alpha-2}  z_a(x)z_b(x), \qquad p_{ab}^+ \Big(  \sum_{j=1}^n z_j(x)\b  z_j \Big)= 2 ( z_a(x)z_b + z_b(x)z_a).\end{equation}	
\end{proof}

\subsubsection{Transition from $V_\mu$ to $V_{\mu-2\epsilon_1}$}
Recall that the linear functional $\ell \colon \calF_{-n}\to \CC$  from Lemma~\ref{lemma:ell} 
is $SL(n,\RR)$ invariant. Since $F\in \calX_\alpha$ takes values in the complexification of a real vector space, its complex conjugate $\b F$ is well defined. Consequently, the natural pairing 
$$ \calX_{\alpha}^\pm\times  \wt\calX_{-n-\alpha}^\pm \to \calF_{-n}$$
can be composed with $\ell$ to obtain  
a sesquilinear form on 
$\calX_{\alpha}^\pm\times  \wt\calX_{-n-\alpha}^\pm$, 
$$\langle F,G\rangle = \ell( \langle \b F(x) ,  G(x)\rangle  ),$$
that is nondegenerate and  $SL(n,\RR)$ invariant.

We denote by $\pi_\mu \colon (\calX_\alpha^\pm)_K\to V_\mu$ the projection onto the $K$-type $V_\mu$. 

\begin{lemma}\label{lemma:-2}In the following, $h_0(k)$ and $h_1(k)$  denote  nonzero numbers that do not depend on $\alpha$. 
	For $k\geq 3$,  \begin{align*}
		\pi_{\mu-2\epsilon_1} (p_{11}^- \xi_k)  &=  (\alpha +k+n-3)h_0(k) \xi_{k-2}\\
		\pi_{\mu-2\epsilon_1} (p_{11}^- \eta_{k})  &=  (\alpha +k+n-2)h_1(k) \eta_{k-2}, 
	\end{align*} 
where $\mu=(k,0,\ldots, 0)$ and $\mu=(k,1,0,\ldots, 0)$, respectively.
\end{lemma}
\begin{proof}
Let us write $\wt \xi_k$ for the highest weight vector in $\wt \calX_{-n-\alpha}$ of highest weight $k\epsilon_1$. 

Note that if we write $p_{11}^+ = X+iY$ with $X,Y\in \mathfrak p$, then $p_{11}^- = X-iY$. 
Since $p_{11}^-\in \mathfrak{sl}(n)_\CC$ and the sesquilinear form is skew-Hermitian for the action of $\mathfrak{sl}(n)_\CC$, 	we have 
	\begin{align*} \langle p_{11}^- \xi_k, \wt \xi_{k-2} \rangle  & = - \langle \xi_k,p_{11}^+ \wt \xi_{k-2}\rangle  \\
	&  =2(\alpha +k+n-3)	\langle \xi_k,\wt\xi_k\rangle.
				\end{align*}  
On the other hand, observe that 
$$  \pi_{\mu-2\epsilon_1} ( p_{11}^- \xi_k) = \pi_{\mu-2\epsilon_1} ( m_\mu( \xi_k \otimes p_{11}^- ))$$
is a weight vector of weight $(k-2)\epsilon_1$ and hence proportional to the highest weight vector $\xi_{k-2}$. If $a(\alpha,k)$ is a constant such that $ \pi_{\mu-2e_1} ( p_{11}^- \xi_k) = a(\alpha,k) \xi_{k-2}$, then 
$$  \langle p_{11}^- \xi_k, \wt \xi_{k-2} \rangle  = a(\alpha,k)\langle \xi_{k-2},\wt\xi_{k-2}\rangle.$$
It follows that 
$$ a(\alpha,k)= (\alpha +k+n-3) h_0(k),$$ 
where 
$$h_0(k)=  \frac{2	\langle \xi_k,\wt\xi_k\rangle}{\langle \xi_{k-2},\wt\xi_{k-2}\rangle }.$$
Since the sesquilinear form is nondegenerate and invariant, the restriction to each $K$-type $V_\mu \otimes \wt V_\mu$ is nondegenerate. Since
$V_\mu \otimes \wt V_\mu\cong \rho^\mu\otimes \rho^\mu$ and the space of $K$ invariant sesquilinear forms on $\rho^\mu$ is $1$-dimensional, we conclude that $\langle \xi_k,\wt\xi_k\rangle$ is nonzero.  Moreover, inspecting the formulas for the highest weight vectors shows that $\langle \xi_k,\wt\xi_k\rangle$ 
is  independent of $\alpha$. 

A parallel reasoning proves the second identity.
\end{proof}

\subsubsection{Transition between special $K$-types}

In this section we will use the Schur functors $\SS_\lambda V$ and $\SS_{[\lambda]} V$ from \cite{FultonHarris} to obtain irreducible representations explicitly as certain spaces of tensors.  Instead of recalling the general construction here, we will describe it only for $\lambda=(k,1)$, which will be the case relevant to us. For $\lambda=(k,1)$, the Young symmetrizer $c_\lambda \in \End(V^{\otimes (k+1)})$ is given by 
$ c_\lambda = a_\lambda \cdot b_\lambda$ with 
$$ a_\lambda = \sum_{\substack{\sigma\in \mathfrak S_{k+1},\\ \sigma(k+1)=k+1}} e_\sigma, \qquad b_\lambda = 1 - e_{(1\ k+1)},$$
where  $(w_1\otimes \cdots \otimes w_{k+1} ) \cdot e_\sigma = w_{\sigma(1)}\otimes \cdots \otimes w_{\sigma(k+1)}$. The image of $c_\lambda$ is denoted by $\SS_\lambda V$ and is an irreducible representation of $GL(n,\CC)$.  
The subspace $\SS_{[\lambda]} V\subset \SS_\lambda V$ consists of the completely trace-free tensors,  defined as  the intersection of the kernels of  all contractions 
$$ w_1\otimes   \cdots \otimes w_{k+1} \mapsto  Q(w_j,w_l)w_1\otimes \cdots \otimes \widehat w_j \otimes \cdots \otimes \widehat w_l \otimes \cdots \otimes w_{k+1}. $$ 
The space $\SS_{[k,1]} V$ is an irreducible representation of $K=O(n)$. Its restriction to $SO(n)$ is irreducible with highest weight $\lambda = k\epsilon_1 +\epsilon_2$, provided $n\geq 5$. For $n=4$, it is reducible and decomposes as $\pi^{\lambda} \oplus \pi^{\lambda^-}$. We denote by $\pi_{(k,1)}$ the projection $\SS_{(k,1)} V\to  \SS_{[k,1]} V$.

\begin{lemma}\label{lemma:hwv-Sym2} Let $n\geq 4$ and $k\geq 1$. A highest weight vector in 
\begin{enuma} 

	\item   $\Sym^k V\otimes \Sym^2 V$ of   weight $(k+1)\epsilon_1 + \epsilon_2$ is
	   $$ \nu_{(k+1,1)}= v_1^k \otimes v_1v_2 - v_1^{k-1} v_2\otimes v_1^2$$

	\item
$\SS_{[k,1]} V \otimes \Sym^2 V$ of weight $(k+1)\epsilon_1$ is
$$ \nu_{(k+1,0)}= \big(\pi_{(k,1)}\otimes \id\big)  \Big ( 
\sum_{j=1}^{n}  \big (v_1^{\otimes k}\otimes   v_j - v_j\otimes v_1^{\otimes k}  \big) \otimes v_1 \b v_j  \Big).$$
Explicitly,
\begin{align*}  \nu_{(2,0)} & = \sum_{j=1}^{n} \big( v_1^{\otimes 2} \otimes  v_j -   v_j\otimes  v_1^{\otimes 2} \big)  \otimes v_1 \b  v_j 
	  + \frac{1}{n-1} ( \wt Q\otimes    v_1- v_1\otimes \wt Q)\otimes v_1v_1\\
	\intertext{and}
\nu_{(3,0)} & = \sum_{j=1}^{n}  \big( v_1^{\otimes 3} \otimes  v_j -   v_j\otimes  v_1^{\otimes 3} \big)  \otimes v_1 \b  v_j  + \frac{1}{n} \Big( \wt Q\otimes v_1^{\otimes 2}  + \sum_{j=1}^n v_j  \otimes v_1  \otimes  \b v_j \otimes v_1  \\
& \quad   -\sum_{j=1}^n v_1\otimes v_j \otimes v_1 \otimes  \b v_j
 - v_1^{\otimes 2} \otimes \wt Q \Big) \otimes v_1v_1
\end{align*}
where $\wt Q=\sum_{j=1}^n v_j \otimes \b v_j \in\Sym^2 V$ is, up to normalization, the metric tensor.
\end{enuma}
\end{lemma}

\begin{proof} We only need to check that the given vectors are annihilated by  the elements $e_{ab}$  and $s_{ab}^+$  for $1\leq a < b\leq m$ and, if $n$ is odd, by  $f_a^+$ for $1\leq a\leq m$.  This is straightforward for  (a). For (b) this is also easy, since 
	$c_{(k,1)} (v_1^{\otimes k} \otimes   v_j)$ 
	is proportional to 
$$v_1^{\otimes k}\otimes  v_j - v_j\otimes v_1^{\otimes k}, $$
$\pi_{(k,1)}$ and $c_{(k,1)}$ are $O(n)$ equivariant,    
	 and said Lie algebra elements annihilate 
	$$ \sum_{j=1}^{n}  \big (v_1^{\otimes k} \otimes v_j \big) \otimes v_1 \b v_j  \in V^{\otimes k} \otimes \Sym^2 V.$$

It remains to find explicit expressions for $\nu_{(2,0)}$ and $\nu_{(3,0)}$. 
We discuss this  only for $\nu_{(3,0)}$ in detail, since the argument for  $\nu_{(2,0)}$ is parallel and simpler.

Denoting by $\operatorname{ctr}_{3,4} \colon V^{\otimes 4}\to V^{\otimes 2}$ the contraction of the third and fourth factor, we obtain a $K$ equivariant map to the space of trace-free $2$-tensors,
$$ \SS_{(3,1)} V\to\SS_{[2]} V \oplus \Lambda^2 V.$$
Here $\SS_{[2]} V\subset \Sym^2 V$ denotes the subspace of trace-free tensors. On the other hand, the $K$ equivariant map  $\psi \colon V^{\otimes 2}\to \SS_{(3,1)} V$, $ v\otimes w\mapsto c_{(3,1)}(\wt Q\otimes v\otimes w)$, satisfies 
$$ \operatorname{ctr}_{3,4} \circ \psi|_{\SS_{[2]} V} = -4n \id_{\SS_{[2]} V}, \quad 
		\operatorname{ctr}_{3,4} \circ \psi|_{\Lambda^2 V} = 4(n+2) \id_{\Lambda^2 V}.$$
Since $n\geq 4$, the Littlewood restriction rule (see, e.g., \cite[Theorem 1.1]{Howe-etal:stable}) implies 
$$ \SS_{(3,1)} V|_K = \SS_{[3,1]}V \oplus \SS_{[2]} V \oplus \Lambda^2 V$$
and hence $\SS_{[3,1]} =\ker( \operatorname{ctr}_{3,4})$.
Since $\operatorname{ctr}_{3,4} \big( v_1^{\otimes 3} \otimes \b v_1 - \b v_1 \otimes  v_1^{\otimes 3}\big)= 2v_1^2$, we conclude that 
$$   \pi_{(3,1)}\big( v_1^{\otimes 3} \otimes \b v_1 - \b v_1 \otimes  v_1^{\otimes 3} \big) =  v_1^{\otimes 3} \otimes \b v_1 - \b v_1 \otimes  v_1^{\otimes 3}   + \frac{1}{2n} \psi( v^2_1). $$

For  $\nu_{(2,0)}$,  contracting the factors $1$ and $2$, we obtain an equivariant map $ \SS_{(2,1)} V\to V$. Comparing dimensions shows that $\SS_{(2,1)} V = \SS_{[2,1]} V \oplus V$. The rest of the argument remains essentially the same.
\end{proof}

It remains to explicitly describe $K$ equivariant isomorphisms from the representations appearing in Lemma~\ref{lemma:hwv-Sym2} to the $K$ types $V_{\mu}$ of $\calX^{\pm}_\alpha$  and to evaluate  the module map $m_\mu \colon V_\mu\otimes \mathfrak p_\CC\to (\calX_\alpha^\pm)_K$ on the highest weight vectors.
We will use the following $K$ equivariant isomorphisms:
\begin{enuma}

	\item $F\colon \SS_{[k]} V\to V_k$ defined on $\Sym^k V$ by 
	$$ F(v_{j_1}\cdots  v_{j_{k}}) =    r^{\alpha-k+1}  d(z_{j_1}\cdots  z_{j_{k}}) -   k r^{\alpha-k-1} 
	z_{j_1}\cdots  z_{j_{k}} \sum_{j=1}^n  z_j(x) \b z_j ,$$
	where 
	$$d(z_{j_1}\cdots  z_{j_{k}}) = \sum_{l=1}^k z_{j_1}(x) \cdots \widehat{ z_{j_l}(x)} \cdots z_{j_k}(x)  z_{j_l}$$ is the exterior derivative.
\item $G\colon \SS_{[k,1]} V\to W_k$ defined as the restriction of $V^{\otimes (k+1)}\to \calV_\alpha$,  
$$  v_{j_1}\otimes \cdots \otimes v_{j_{k+1}} \mapsto  r^{\alpha-k}z_{j_1}(x)\cdots z_{j_k}(x) z_{j_{k+1}}.$$

\end{enuma}

\begin{lemma}\label{lemma:special-K} Let $n\geq 4$ and $k\geq 1$. Then   
\begin{align*}
	  m_{(k,0)} \big( (F\otimes T) (\nu_{(k+1,1)})\big)   & =2 (\alpha +1) \eta_{k+1},\\ 
	  m_{(2,1)} \big( (F\otimes T) (\nu_{(3,0)})\big)   & =  c_1(n)
	   (\alpha -2) \xi_3, \\
	    m_{(3,1)} \big( (F\otimes T) (\nu_{(4,0)})\big)   & = c_2(n)(\alpha -3) \xi_4,
\end{align*} 
where $c_1(n), c_2(n)$ are  numbers depending only on $n$.

\end{lemma}
\begin{proof} Observe that 
	$$ q_{ab}(r^\alpha)=  2 \alpha r^{\alpha-2}  z_a(x)\b z_b(x).$$	
Using this and \eqref{eq:p+r}, evaluating $m_\lambda$ on the highest weight vectors is  straightforward.
\end{proof}

\begin{remark}
	The presence of the factor $\alpha+1$ in Lemma~\ref{lemma:special-K} hints at the existence of a $GL(n,\RR)$ invariant subspace for the value $\alpha=-1$. Such a subspace indeed exists: Since $df\in \calX_{-1}$ for $f\in \calF_0$, the image of the differential $d\colon \calF_0\to \calX_{-1}$ yields a proper invariant subspace.
\end{remark}

Finally, we consider $n=3$. In this special dimension, the Hodge star gives rise to an $SO(3)$ equivariant map,  
$$\calF_{\alpha} \to \calX_{\alpha}, \quad f \mapsto   i_E ({*} df) ,$$
where $E(x)= x_1 \frac{\partial}{\partial x_1} + x_2 \frac{\partial}{\partial x_2} +x_3 \frac{\partial}{\partial x_3}$ is the Euler vector field. One immediately verifies that 
$$ i_E ({*}dz_1) = \frac{i}{\sqrt 2} (z_1 dz_2 -z_2 dz_1), \quad i_E ({*} dz_3) =  \frac{i}{\sqrt 2} ( z_2 dz_3-z_3 dz_2), \quad i_E (* dz_2)= -\frac{i}{\sqrt 2} 
(z_1 dz_3 -z_3  dz_1)$$
This leads us to define a map $\Sym^k V\to \calX_\alpha$ by 
\begin{align*}  v_{j_1}\cdots  v_{j_{k}} \mapsto 
& r^{\alpha-k} \Big(  \frac{\partial f}{\partial z_1}  (z_2(x) z_1 -z_1(x)  z_2) - \frac{\partial f}{\partial \b z_1}  (z_2(x) \b z_1 - \b z_1(x)  z_2) + \frac{\partial f}{\partial z_2}  (z_1(x) \b z_1 -\b z_1(x)  z_1)
\Big),
\end{align*}
where $f= z_{j_1}\cdots  z_{j_{k}}$. 
Observe that by restriction we obtain a map $H\colon \SS_{[k]} V\to W_{k}$. As above, let $F\colon  \SS_{[k]} V\to V_{k}$ denote the map induced by the differential. 

\begin{lemma}\label{lemma:n=3} Let $n=3$ and $k\geq 1$. The vector 
	$$u_k =  v_1^k \otimes v_1v_2 - v_1^{k-1} v_2\otimes v_1^2 $$
	in $\Sym^k V \otimes \Sym^2 V$
	is a highest weight vector of highest weight $(k+1)\epsilon_1$. Put $\mu= k \epsilon_1$. 
Then 
\begin{align*} 
	m_{\mu  } ( (F\otimes T)(u_k))  &=2(\alpha +1 ) \eta_{k+1},\\
m_{\mu} ( (H\otimes T)(u_k))  &= (\alpha -k)  \xi_{k+1}.
\end{align*} 
\end{lemma}
\begin{proof}
We omit the straightforward computation.
\end{proof}

\subsection{Decomposition series for $\alpha=-(n+1)$}
Throughout this subsection let $n\geq 2$. 
Consider the subspace $\calY \subset \calX_{-(n+1)}$ defined by 
$$ \ell ( \xi F)=0$$
for every linear functional $\xi\colon \RR^n\to \CC$.  Let $\calY^\pm = \calX_{-(n+1)}^\pm\cap \calY$ denote the subspaces of even and odd elements under the action of $\pi(-\id)$. 

\begin{lemma} \label{lemma:yk} For $n\geq 3$, 
	$$\calY^- = \calX_{-(n+1)}^- ,\quad   \calY^+_K = \bigoplus_{l\geq 2}  V_{2l} \oplus W_{2l-1}.$$ 

\end{lemma} 
\begin{proof} Observe that $F\in \calY$ if and only if $\ell(\xi \langle F, y\rangle)=0$ for all $\xi \in V^*$ and $y\in V$. 
	The map $\calX_{-(n+1)}\to V\otimes V^*$, $F\mapsto [ \xi\otimes y\mapsto \ell(\xi\langle F,y\rangle)]$, is $K$ equivariant. Since $V\otimes V^*= \CC \oplus \rho^{(2)} \oplus \rho^{(1,1)}$, Schur's lemma implies the first statement and  
	$\calY^+_K \supseteq \bigoplus_{l\geq 2}  V_{2l} \oplus W_{2l-1}$.
	Inspecting the highest weight vectors $\xi_2$ and $\eta_1$ shows that the inclusion is actually an equality.
\end{proof}

Denote by $\calG\subset \calF_{-n}$ the kernel of $\ell$ and let $\calG^\pm$ denote the subspaces of even and odd functions. 
\begin{lemma}\label{lemma:calG} For $n\geq 2$, the  representation of  $GL(n,\RR)$ on $\calG^\pm$ is irreducible.
\end{lemma}
\begin{proof}
	 By  \cite[Theorem 3.4.3]{HoweLee:degenerate}, the $GL(n,\RR)$ representation $\calF_0^-$ is irreducible.
	 On the other hand, $\calF_0^+$ obviously contains the invariant subspace of constant functions denoted $R^+(0,0)$.  According to \cite[Theorem 3.4.2]{HoweLee:degenerate}, the quotient $\calF_0^+/R^+(0,0)$ is irreducible. 
	 Via duality induced by $\ell$ between $\calF_{-n}$ and $\calF_0$, we conclude that  $\calG^\pm$ is irreducible. 
\end{proof}

The following lemma takes care of the $2$-dimensional case of the theorem we wish to prove.

\begin{lemma}\label{lemma:dim2}
	If $n=2$, then the $GL(2,\RR)$ representations $\calX_{\alpha}^\pm$ and $\calF_{\alpha-1}^\pm \otimes \Lambda^2 V^*$ are isomorphic. Under this isomorphism, $ \calY^\pm$ corresponds to the functions in  $\calF_{-4}^\pm \otimes \Lambda^2V^*$ satisfying, for $1\leq i, j\leq 2$, 
	\begin{equation}\label{eq:dim2-cond} \int_{S^1} f(u) u_iu_j \, du=0.\end{equation}
	Consequently, $\calY^\pm$ is irreducible.
\end{lemma}

\begin{proof}
	Consider the linear  map $\varphi\colon V^*\to V\otimes \Lambda^2 V^*$ defined by 
	$$ \xi \mapsto \big( \xi(u)v - \xi(v)u\big)\otimes \omega,$$
	where $\omega$ is a nonzero $2$-form and $u,v\in V$ are vectors satisfying $\omega(u,v)=1$. This definition is independent of  the choice $\omega$, and the linear map thus defined is $GL(2,\RR)$ equivariant.
	
	If $\xi(x)=0$ for some nonzero $x\in \RR^2$, then $\varphi(\xi)$ is proportional to $x\otimes \omega$. Consequently, $F\mapsto \varphi \circ F$ defines a $GL(2,\RR)$ isomorphism  
	$ \calX_\alpha^\pm \to \calF_{\alpha-1}^\pm \otimes \Lambda^2 V^*$.  Moreover, under this isomorphism, the subspace  $ \calY^\pm$ corresponds to the functions in  $\calF_{-4}^\pm \otimes \Lambda^2V^*$ satisfying \eqref{eq:dim2-cond}. According to \cite[Theorem 3.4.3]{HoweLee:degenerate}, $\calF_2^-$ is irreducible. By \cite[Theorem 3.4.2]{HoweLee:degenerate},  $\calF_2^+$ has a finite-dimensional invariant subspace $R^+(2,0)$, corresponding to $2$-homogeneous polynomials, and $\calF_2^+/R^+(2,0)$ is irreducible. Via duality induced by $\ell$ between $\calF_{-4}$ and $\calF_2$, we conclude that  $\calY^\pm$ is irreducible. 	
\end{proof}

\begin{theorem} \label{thm:irred} 
For $n\geq 2$, the representation of  $GL(n,\RR)$ on $\calY^\pm$ is irreducible.
\end{theorem} 
\begin{proof} Since Lemma~\ref{lemma:dim2} takes care of $n=2$, we  can assume $n\geq 3$. 
For the value $\alpha = -(n+1)$ the numbers appearing in Lemmas~\ref{lemma:+2}, \ref{lemma:-2}, \ref{lemma:special-K} and \ref{lemma:n=3} are nonzero. 
They could be zero in Lemma~\ref{lemma:-2} for $\xi_4$ and $\eta_3$, but precisely  these cases are irrelevant by Lemma~\ref{lemma:yk}. 
It follows that the $(\fgg, K)$ modules 
$ \calY^\pm_K$ are irreducible. 
\end{proof}

\begin{proof}[Proof of Theorem~\ref{thm:length2}]
Consider the $GL(n,\RR)$ equivariant linear map  $\rho \colon \calV_\alpha^\pm \to \calF_{\alpha+1}^\pm$ defined by  $\rho(F)(x)= \langle F(x),x\rangle$.  Observe that  $\ker \rho =\calX_{\alpha}^\pm$.  

Suppose now that $\alpha= -(n+1)$. Recall  that  $\calU \subset \calV_{-(n+1)}$ is the subspace  defined by 
$ \ell ( \xi F)=0$
for every linear functional $\xi\colon \RR^n\to \CC$.
Hence $\rho$ restricts to a nonzero map $\calU^\pm\to \calG^\pm$ with kernel $\calY^\pm$. 
Since $\calG^\pm$ is irreducible by Lemma~\ref{lemma:calG}, the Casselman--Wallach theorem (Theorem~\ref{thm:CW}) implies that $\calU^\pm$ fits into an exact sequence 
$$ 0 \to \calY^\pm \to  \calU^\pm \to \calG^\pm\to 0.$$
Since $\calY^\pm$ is irreducible by Theorem~\ref{thm:irred}, we conclude that $\calU^\pm$ has length $2$. 	
\end{proof}

\section{Applications}

\subsection{The probability kernel property of curvature measures}
The set of all Borel functions admits the following inductive description. 
For every countable ordinal number $\alpha<\omega_1$, the set $\calB_\alpha$ of Baire functions of class $\alpha$ is defined as follows:
$\calB_0$ is the set of continuous functions $f\colon \RR^n\to \RR$. If $\alpha>0$, then $\calB_\alpha$ consists of all pointwise limits of sequences $(f_i)$ of functions  $f_i\in \calB_\gamma$ with $\gamma<\alpha$. 
A classical theorem asserts  that 
$$ \bigcup_{\alpha<\omega_1} \calB_{\alpha}$$
coincides with the set of all Borel functions, see \cite[Theorem 24.3]{Kechris:dst}.
\begin{lemma} \label{lemma:bounded-baire} Let $0<\alpha < \omega_1$. If $f\in \calB_\alpha$ is  bounded, $\sup_{\RR^n} |f|\leq C$, then $f$ is the limit of a sequence of functions $f_i\in\calB_\gamma$,  $\gamma <\alpha$, satisfying $\sup_{\RR^n}|f_i|\leq C$. 
\end{lemma} 
\begin{proof} 
Transfinite induction shows that if $f,g\in \calB_\gamma$, then  $\max(f,g)$ and $ \min (f,g)$ belong to the same Baire class. 
\end{proof}

\begin{proof}[Proof of Proposition~\ref{prop:baire}] If $f$ is a bounded Borel function, then $f\in \calB_\alpha$ for some $\alpha <\omega_1$. We now argue by transfinite induction. If $\alpha=0$, then $f$ is bounded and continuous. By definition, 
	$K\mapsto \Phi(K,f)$ is continuous, hence in particular Borel.  For the inductive step, suppose $ K\mapsto \Phi(K,g)$ is Borel for all bounded functions $g\in \calB_\gamma$ with $\gamma<\alpha$. By Lemma~\ref{lemma:bounded-baire}, a bounded function $f\in \calB_\alpha$ 
	is the limit of a sequence of functions $f_i\in\calB_\gamma$,  $\gamma <\alpha$, satisfying $\sup_{\RR^n}|f_i|\leq C$ for all $i$. By the dominated convergence theorem, 
	$$ \Phi(K,f)= \lim_{i\to \infty}  \Phi(K, f_i)$$
	and hence $ K\mapsto \Phi(K,f)$ is Borel.
\end{proof}

\subsection{Characterization of Federer's curvature measures}

We say that a   curvature measure $\Phi\in \Curv(\RR^n)$ is $SO(n)$ invariant if   $$k\cdot \Phi = \Phi\quad \text{for all }k\in SO(n).$$

\begin{proposition} \label{prop:Z-invariant} Let $n\geq 2$. If $\Phi \in \Z(\RR^n)$ is $SO(n)$ invariant, then $\Phi=0$. 
\end{proposition}

\begin{proof}
	Without loss of generality, we may assume that $\Phi$ is homogeneous of degree $i$. If $i=n$, there is nothing to prove. If $i=n-1$, then the claim follows from Theorem~\ref{thm:deg-n-1}. 
	
	Suppose that $i=n-2$. Then $B(\Phi)$  is homogeneous of degree $n-1$ and  $SO(n)$ invariant. Consequently, there exists a continuous  function $f\colon S^{n-1}\to \CC^n$ satisfying \eqref{eq:B-f} and  $k f(k^{-1}u)= f(u)$ for every $k\in SO(n)$ and $u\in S^{n-1}$. It follows that $\langle f(u),u\rangle =c  $  is  constant. Replacing $f$ by $f(u)-cu$ if necessary, we may assume that, for all $u\in S^{n-1}$, 
	$$ \langle f(u),u \rangle =0.$$
	Since $f(u)$ is invariant under the stabilizer subgroup of $u$ in $SO(n)$, we conclude that $f(u)=0$ provided $n\geq 3$. 
	 If $n=2$, then $SO(2)$ invariance  implies that $f$ is proportional to $\rho_{\pi/2} u$. In any case, $B(\Phi)=0$ and hence $\Phi=0$ by the injectivity of the Bernig embedding.
	
	The degrees of homogeneity  treated so far imply the proposition  in particular in dimension $n=2$. Now let $n\geq 3$ and assume that the proposition has been proved in dimension $n-1$. Let $H$ be a linear hyperplane in $\RR^n$ and denote by $\Phi|_H\in \Curv(H)$ the restriction of $\Phi$  to convex bodies in $H$. Since $\Phi|_H$ is $SO(H)$ invariant, the inductive hypothesis implies $\Phi|_H=0$. We conclude that  $\Phi$ is simple. The characterization of simple curvature measures (Theorem~\ref{thm:simple}) and the degrees of homogeneity treated so far imply $\Phi=0$, as desired.
\end{proof}

\begin{proof}[Proof of Theorem~\ref{thm:federer}] By Hadwiger's characterization of linear combinations of the  intrinsic volumes as the  only $SO(n)$ invariant elements in $\Val(\RR^n)$, there exist  complex numbers $a_0,\ldots, a_n$ such that 
	$$ \glob(\Phi)= \sum_{i=0}^n a_i \mathrm V_i,$$
where $\mathrm V_i = \frac{1}{\vol(S^{n-i-1})} \binom{n-1}{i} \glob \C_i$ are the intrinsic volumes. It follows that $\Phi- \sum_{i=0}^n a_i \C_i$ belongs to $\Z(\RR^n)$. Proposition~\ref{prop:Z-invariant} implies
$\Phi- \sum_{i=0}^n a_i \C_i=0$, which concludes the proof.  
\end{proof}

\bibliographystyle{abbrv}
\bibliography{ref}

\end{document}